\documentclass[a4paper,10pt]{article}

\RequirePackage[OT1]{fontenc} \RequirePackage{amsthm,amsmath}
\RequirePackage[longnamesfirst]{natbib}
\usepackage{bm}
\usepackage{graphics,epsfig,epsf,psfrag}
\usepackage{mathrsfs}
\usepackage{url}
\usepackage{color, appendix}
\usepackage{graphicx}
\usepackage{amsmath, amsfonts, amsthm,amssymb}
\usepackage{bbm}
\usepackage{hyperref}
\usepackage{latexsym,amsfonts,amssymb,amsthm,amsmath}
\usepackage{multirow}
\usepackage{caption}
\usepackage{enumerate, framed}
\usepackage{graphicx,lscape,rotate}
\usepackage{algorithmic, algorithm}
\usepackage{latexsym,amsfonts,amssymb,amsthm,amsmath}
\usepackage{multirow}
\usepackage{graphicx, subfig}
\usepackage{caption}
\usepackage{gensymb}

\usepackage[a4paper]{geometry}
\usepackage{amsmath}
\usepackage[latin1]{inputenc}
\usepackage[T1]{fontenc}
\allowdisplaybreaks[2]
\usepackage{graphicx}
\usepackage{amssymb}
\usepackage{amsthm}
\usepackage{authblk}

\newtheorem{theorem}{Theorem}
\newtheorem{remark}{Remark}
\newtheorem{definition}{Definition}
\newtheorem{proposition}{Proposition}
\newtheorem{lemma}{Lemma}

\newtheorem{corollary}{Corollary}
\renewenvironment{proof}{\noindent{\bf Proof.}}{\hfill
  $\blacksquare$\par\noindent}
 \newcommand{\com}[1]{}

\newcommand{\bx}{\bm x}
\newcommand{\bY}{\bm Y}
\newcommand{\bX}{\bm X}
\newcommand{\bbeta}{\bm \beta}

\newcommand{\bSigma}{\bm \Sigma}
\newcommand{\bTheta}{\bm \Theta}
\newcommand{\mE}{\mathbb E}
\newcommand{\mP}{\mathbb P}
\newcommand{\bc}{\bm c}

\newcommand{\sign}{\operatorname{sign}}

\graphicspath{{./fig/}}

\title{Directional FDR Control for Sub-Gaussian Sparse GLMs}
\vspace{.5cm}
\makeatletter
\def\timenow{\@tempcnta\time
  \@tempcntb\@tempcnta
  \divide\@tempcntb60
  \ifnum10>\@tempcntb0\fi\number\@tempcntb
  \multiply\@tempcntb60
  \advance\@tempcnta-\@tempcntb
  \ifnum10>\@tempcnta0\fi\number\@tempcnta}
\makeatother
\author{Chang Cui\thanks{School of Mathematical Sciences and Center for Statistical Sciences, Peking University, Beijing, China. Email: ca.chcui@pku.edu.cn. The authors are in alphabetical order contributed equally to this work.}, Jinzhu Jia \thanks{School of Public Health and Center for Statistical Science, Peking University, Beijing, China. Email: jzjia@pku.edu.cn}, Yijun Xiao \thanks{School of Mathematical Sciences, Peking University, Beijing, China. Email:xiaoyijun@pku.edu.cn}, Huiming Zhang \thanks{1.Department of Mathematics,
University of Macau, Taipa Macau, China; 2.UMacau Zhuhai Research Institute, Zhuhai, China. Email: huimingzhang@um.edu.mo}, }

\date{}

\usepackage{geometry}
\begin{document}

\maketitle

\begin{abstract}
High-dimensional sparse generalized linear models (GLMs) have emerged in the setting that the number of samples and the dimension of variables are large, and even the dimension of variables grows faster than the number of samples. False discovery rate (FDR) control aims to identify some small number of statistically significantly nonzero results after getting the sparse penalized estimation of GLMs. Using the CLIME method for precision matrix estimations, we construct the debiased-Lasso estimator and prove the asymptotical normality by minimax-rate oracle inequalities for sparse GLMs. In practice, it is often needed to accurately judge each regression coefficient's positivity and negativity, which determines whether the predictor variable is positively or negatively related to the response variable conditionally on the rest variables. Using the debiased estimator, we establish multiple testing procedures. Under mild conditions, we show that the proposed debiased statistics can asymptotically control the directional (sign) FDR and directional false discovery variables at a pre-specified significance level. Moreover, it can be shown that our multiple testing procedure can approximately achieve a statistical power of 1. We also extend our methods to the two-sample problems and propose the two-sample test statistics. Under suitable conditions, we can asymptotically achieve directional FDR control and directional FDV control at the specified significance level for two-sample problems. Some numerical simulations have successfully verified the FDR control effects of our proposed testing procedures, which sometimes
outperforms the classical knockoff method.
\end{abstract}



\section{Introduction}

For the past few years, in bioinformatics, signal processing, finance, and many other scientific fields, large-scale datasets with multiple variables have been often involved. When faced with such high-dimensional data, in regression problems, traditional statistical methods usually cannot perform well. This phenomenon is also known as the curse of dimensionality. Based on a priori assumption that the true model should be simple and only related to a small set of features, we can choose a better model through dimensionality reduction. A popular method is to penalize the original objective function by adding a penalty (or regularization term) to select effective variables. The groundbreaking work is Lasso \citep{tibshirani1996regression}, which adds the penalty of $\ell_1$-norm to the least-squares to achieve sparsity. Since Lasso was proposed, statisticians have developed various regularization methods based on different penalty and loss functions; see \cite{fan2020statistical} for a comprehensive review and detailed introductions. In many situations, the linear model is not enough to characterize the data. One of the influential generalizations is generalized linear models \citep[GLMs,][]{nelder1972generalized}, which allows for response variables with distributions other than a normal distribution and thus has a wider range of applications. GLMs have many important developments in statistical inference and computing in the past half-century. The theory for high-dimension GLMs estimation has been studied before, such as logistic regression \citep{bunea2008honest,huang2021weighted}, Poisson regression \citep{blazere2014oracle} and negative binomial regression \citep{zhang2022elastic}.

Meanwhile, the problem of statistical inference in high-dimensional models is also very critical since scientists want to select the most influential factors and want to know which predictors have non-zero influence and which ones do not. This is actually a question of multiple testing. For each covariate, we hope to determine whether it is important or not at the same time. To formulate this problem precisely, we need a criterion to assess the result. When the dimensionality is relatively low, a reasonable and commonly used criterion is the family-wise error rate (FWER). For high-dimensional data sets, for example, among millions of genes, scientists want to study which of them are related to a specific disease. In this case, the family-wise criterion would be somewhat conservative. In contrast, false discovery rate (FDR) proposed by \cite{BH95} is a criterion more suitable for the large-scale multiple hypothesis testing problem. After doing some penalized estimation for regressions, FDR control methods can typically be applied via p-value thresholds, such as BH procedure \citep{BH95}. BH procedure was proved to be able to control FDR in independent or positive dependent cases \citep{BH95,benjamini2001control}. Subsequent research \citep{liu2014phase} made a series of methodological and theoretical improvements. Afterwards, \cite{barber2015controlling} originally proposed the knockoff method to control FDR for linear models. This method constructs the knockoff variables as the original covariates' negative controls to identify the truly crucial variables. \cite{candes2018panning} and \cite{fan2020rank} further generalized it to high-dimensional nonlinear models. Another line for the inference of high-dimensional settings is to construct debiased estimators \citep{zhang2014confidence,javanmard2014confidence,van2014asymptotically,javanmard2019false,yuyi2018}. These methods start from the regularized estimator, use different techniques to construct the debiased estimator, and then draw statistical inferences on the asymptotic normality nature of it.

Our goal is to study the question that controls the directional FDR of estimators for high-dimensional GLMs. This statistical controlling problem was first studied in \cite{javanmard2019false} for linear regression models, which gave an debiased estimator for the high-dimensional linear regression model of the random design and proposed a procedure to control directional FDR based on the debiased estimator. \cite{xia2018joint} constructed the debiased estimator based on the inverse regression to handle the joint test of high-dimensional multiple regressions and control FDR. Our goal is to see whether the directional FDR can be obtained for the complex statistical model such as GLMs. In this paper, we concentrate mainly on the class of sparse estimators from GLM-like estimating equations. In future studies, our framework could be extended to more complex models such as survival models and Ising models with dependent responses.

Recently, \cite{ma2020global} used low dimensional projection to construct debiased estimators. They proposed the logistic multiple testing (LMT) procedure for large-scale multiple testing in high-dimensional logistic regression, which can asymptotically control FDR and the number of falsely discovered variables (FDV). However, due to bounded responses, \cite{ma2020global} did not provide methods for other link functions of GLMs. The main contributions of the present article are as below. First, by a new technique, sharper oracle inequalities for Lasso penalized GLM is first obtained, matching the optimal minimax lower bounds $O({s_0}\sqrt {\frac{{\log (p/{s_0})}}{n}} )$ in \cite{ye2010rate}(see the remark before Theorem \ref{lassobeta1}). Second, we propose the GLM multiple testing (GMT) procedures to control directional FDR, the number of directional FDV, and directional FWER using the debiased estimator for high-dimensional GLM. We utilize the constrained $\ell_1$-minimization for inverse matrix estimation \citep[CLIME; ][]{cai2011constrained} to calculate the debiased estimator for GLM under random design. We also give the test procedure for the two-sample testing problem. Third, for the required conditions, the asymptotic normality of the debiased estimator is proved under the sparsity parameter condition $s_0=o\left(\frac{n^{1/4}}{\log p}\right)$, which greatly relaxes the sparsity assumption in \cite{xia2020revisit}. And the assumptions for the covariance and precision matrix are more understandable than \cite{ma2020global} and \cite{yuyi2018}.

The rest of this paper is arranged as follows. Section 2 describes model settings and gives the debiased estimator under GLMs and in Section 2 we prove the  estimator's theoretical properties. Section 3 presents the method of controlling FDR for large-scale testing problem and two-sample problem based on the debiased estimator. Section 4 conducts numerical experiments on several methods under different settings which prove the effectiveness of our proposed method. The technical proofs are shown in Supplementary Materials.

\section{Notations, Model Settings and Debiased Lasso}

In this section, we first describe the notations and model settings and then we propose the debiased Lasso estimator and provide its asymptotic properties.

\subsection{Notations}
The notations we shall use in this article are described here. For any positive integer $n$, let $[n]=\{1,2, \ldots, n\} .$
For set $S$, $S^c$ represents its complement, and let $|S|$ denote its cardinality.
Let $\boldsymbol{I}_p$ denote the $p\times p$ identity matrix.
For $d$-dimensional vector $\bm{v}=\left(v_{1}, \ldots, v_{d}\right)^T  \in \mathbb{R}^{d}$, $\|\bm{v}\|_{q} = (\sum_{i=1}^{d}\left|v_{i}\right|^{q})^{1/q}$ represent its $\ell_{q}$ norm and $\|\bm{v}\|_{\infty}:=\max _{i \in [d]}\left|v_{i}\right|.$
Given set $S \subseteq \{1, \ldots, d \}, $ we write $ v_{S}: =
\left(v_{i}\right)_{i \in S} \in \mathbb{R}^{|S|}.$
For matrix $\bm{A}=\left(A_{ij}\right)_{i\in[n],j\in[p]} \in \mathbb{R}^{n\times p}$ and set $S \subseteq [p],$
let $\bm{A}_S\in \mathbb{R}^{n\times |S|}$ denote the matrix only consisting of the columns whose index is in $S$.
Let $\|\bm{A}\|_{\infty}:=\max _{i\in[n],j\in[p]}\left|A_{i j}\right|,$ and
$\|\bm{A}\|_{\mathrm{op}, \infty}:=\sup _{\bm{v} \neq 0}\left(\|\bm{A} \bm{v}\|_{\infty} / \|\bm{v}\|_{\infty}\right),
\quad \|\bm{A}\|_{\mathrm{op}, 1}:=\sup _{\bm{v} \neq 0}\left(\|\bm{A} \bm{v}\|_{1} / \|\bm{v}\|_{1}\right)$
respectively represent the operator norm $\ell_{\infty}$
and $\ell_{1}.$
For a smooth function $f(x)$ defined on $\mathbb{R}$, define $\dot{f}(x)=d f(x) / d x$ and $\ddot{f}(x)=d^{2} f(x) / d x^{2}$.
If $f: \mathbb{R}^{p} \rightarrow \mathbb{R}$ is differentiable, let $\nabla f$ denote the gradient of $f$.
For sequences $\left\{a_{n}\right\}$ and $\left\{b_{n}\right\}$, if $\lim_{n} a_{n} / b_{n}= 0$,
then we say $a_{n}=o\left(b_{n}\right).$ If there is a constant $C$ such that $a_{n} \leq C b_{n}$ holds for any $n$, then we say $a_{n}=O\left(b_{n}\right), \mbox{ written as } a_{n} \lesssim b_{n}$ or $b_{n} \gtrsim a_{n}.$ If $a_{n} \lesssim b_{n}$ and $a_{n} \gtrsim b_{n}$, then say $a_{n} \asymp b_{n}.$
In the same probability space, let $X_{1}, X_{2}, \ldots$ and $Y_{1}, Y_{2}, \ldots$ be random variables.
We write $X_{n}=o_{p}\left(Y_{n}\right)$ if and only if $X_{n} / Y_{n} \stackrel {p }\rightarrow 0$, where $ \stackrel {p }\rightarrow$ means convergence in probability.
We write $X_{n}=O_{p}\left(Y_{n}\right)$ if and only if for any $\varepsilon>0,$ there exists constant $C_{\varepsilon}>0$ such that $\sup _{n} \mP \left(\left\|X_{n}\right\| \geq C_{\varepsilon}\left|Y_{n}\right|\right)<\varepsilon.$
When $\beta^{0} \in \mathbb{R}^{p}$ is the true coefficient vector to be estimated, $S=\operatorname{supp} (\bbeta^{0})= \{j \in[p]:\beta_{j}^{0} \neq 0\}$ denote the nonzero elements in $\bbeta^{0}$.
Let $s_0=|S|$ denote the number of non-zero elements in $\bbeta^{0}$.

\subsection{Model settings: Sub-Gaussian GLMs}\label{subsec:model}

In this article, the design matrix $\bX=\left(\bx_{1}, \ldots, \bx_{n}\right)^{T} \in \mathbb{R}^{n \times p}$ is random. The response variable is written in a vector $\bY=\left(y_{1}, \ldots, y_{n}\right)^{T} \in \mathbb{R}^{n}$. The data $\{(\bx_i, y_i)\}_{i=1}^{n}$ are the $n$ i.i.d observations. Each $y_i$ given $\bx_i$ follows a GLM if
\begin{equation}\label{glmlike}
f(y_i|\bx_i)=\exp \{\frac{y_{i} \bx_{i}^{T} \bbeta^0-b\left(x_{i}^{T} \bbeta^0\right)}{a(\psi)}+c\left(y_{i} ; \psi\right)\},
\end{equation}
where $\psi \in \mathbb{R}>0$ is fixed and it is a known scale parameter; $ \bbeta^0 \in \mathbb{R}^{p}$ is the GLM parameter of interest. The term $c\left(y_{i} ; \psi\right)$ is the weight function and ${b}(\theta)$ is the log-partition function. Without loss of generality, we assume that $a(\psi)=1$.

Our goal is to estimate the parameter $ \bbeta^0 $ in GLMs  \eqref{glmlike}  given $n$ i.i.d. samples $\{(\bx_i, y_i)\}_{i=1}^{n} .$ From densities of exponential families random variables, the conditional expectation and variance of the response given the covariates are $\dot{b}( \bx_{i}^{T} \bbeta^0)=\mE (y_i | \bx_i)$ and $\ddot{b}( \bx_{i}^{T} \bbeta^0)=\mathrm{Var}(y_i | \bx_i)$.

Examples of \eqref{glmlike} contain linear models (${b}(t)=t^2/2$), logistic regressions ($b(t)=\log(1+e^t)$)
model, Poisson regressions ($b(t)=e^{t}$), and exponential regressions ($b(t)=-\log(-{t})$), among others. To estimate the unknown parameter $\bbeta^0 \in \mathbb{R}^p$, we consider the log-likelihood function
$\ell_n(\bbeta):=\frac{1}{n}\sum_{i=1}^{n}\left[y_{i} \bx_{i}^{T} \bbeta-b\left(\bx_i^T\bbeta\right)\right]$
and the estimation equation is
\begin{equation}\label{score}
\dot{\ell}_n(\bbeta)=\nabla_{\bbeta}\ell(\bbeta)=\frac{1}{n}\sum_{i=1}^{n}\bx_i[y_i-\dot{b}(\bx_i^T\bbeta)] = \bm 0.
\end{equation}

Note that \eqref{score} is also called the quasi-likelihood estimation equation if we do not assume that the responses belong to exponential family. This means that the use of the estimation equation \eqref{score} only involves the first-order conditional moments of $y_i$. Thus, we do not need to assume that the distribution of $y_i$ belongs to an exponential family. To generalize the assumption \eqref{glmlike}, we first recall the definition of a sub-Gaussian random variable \citep{zhang2020concentration}:

\begin{definition}
A zero-mean random variable $Z$ is said to follow sub-Gaussian with \emph{variance proxy} $\sigma^2 > 0$ (denoted by $Z \sim \mathrm{subG}(\sigma^2)$) if $\mathrm{E}e^{s Z}\le e^{\frac{\sigma^{2} s^{2}}{2}},~\forall s \in \mathbb{R}.$
\end{definition}
The compact parameter space assumption of exponential family leads to the sub-Gaussian family, see Theorem 3.1 in \cite{zhang2020concentration}. In this article, we do not need the assumption that $y_i$ follows the specific distribution defined as \eqref{glmlike}. Conditioning on $\bx_i$, under conditional variance proxy $\sigma_i=:\sigma(\bx_i^T \bbeta^0)$, we only assume the \emph{sub-Gaussian regression models}
\begin{equation}\label{subREG}
y_i - \mathbb E (y_i| \bx_i) \sim \mathrm{subG}(\sigma_i^2),~~i=1,2,\cdots,n,
\end{equation}
with mean $\dot b(\bx_i^T \bbeta^0)$ and variance $\ddot b(\bx_i^T \bbeta^0)$ for a certain function $b(\cdot)$.  Note that from (3.2) in \cite{zhang2020concentration} we know that $\sigma_i \ge \ddot b(\bx_i^T \bbeta^0)$.

\subsection{Regularity conditions for the debiased estimator}\label{subsec:Regularity}
For $\lambda>0$, we let
\begin{equation}\label{lassoes}
\widehat{\bbeta}=\widehat{\bbeta}(\lambda):=
\arg \min _{\bbeta \in \mathbb{R}}
\left\{-\ell(\bbeta)+\lambda\|\bbeta\|_{1}\right\}
\end{equation}
be the Lasso-type estimator of the true parameter defined by $\bbeta^0:=\arg \min _{\bbeta \in \mathbb{R}}
\left\{-\mE \ell(\bbeta)\right\}$ with no penalty. Hence, $\lambda\|\bbeta\|_{1}$ is the bias term in our following analysis, which need to be removed. Many researchers have used debiasing or correction methods \citep[e.g.,][]{van2014asymptotically,yuyi2018} to reduce bias of the regularized estimator and draw statistical inference more accurately. Likewise, we shall construct the debiased estimator for the sub-Gaussion GLMs.

Define the component-wise additive inverse of Hessian matrix
\begin{center}
$\bSigma:=-\mE \{\ddot{\ell}(\bbeta^0)\}=
\frac{1}{n}\mE\sum_{i=1}^{n}\bx_i\bx_i^T\ddot{b}(\bx_i^T\bbeta^0)=\mE\bx_i\bx_i^T\ddot{b}(\bx_i^T\bbeta^0).$
\end{center}
From \cite{van2014asymptotically}, the debiased estimator is constructed as
\begin{equation}\label{debia}
\widehat{\bbeta}^d:=\widehat{\bbeta}+
\widehat{\bTheta}\dot{\ell}(\widehat{\bbeta}).
\end{equation}
Here, $\widehat{\bbeta}$ is the Lasso estimator \eqref{lassoes}, and $\widehat{\bTheta}=(\widehat{\Theta}_{ij})_{i,j=1}^p $
is the estimation of the inverse of $\bSigma$. Let $\widehat{\bTheta}=(\widehat{\bTheta}_{1},\ldots, \widehat{\bTheta}_{p})^T $. In this paper, we use CLIME \citep{cai2011constrained} to derive $\widehat{\bTheta}$ and thus
\begin{equation}\label{Theta}
\widehat{\bTheta}_{j} \in \underset{\boldsymbol{\omega} \in \mathbb{R}^{p}}{\operatorname{argmin}}
\left\{
\left\|\boldsymbol{\omega}\right\|_{1}:\left\|\bSigma_n
 \boldsymbol{\omega}-\boldsymbol{e}_{j}
\right\|_{\infty} \leq \lambda_{n}
\right\},
\end{equation}
where the vector $\boldsymbol{e}_{j}\in \mathbb{R}^{p}$ is the unit vector whose $j$-th element is set to 1 and the other elements are all 0.  $\bSigma_n$ is defined as the sample version of $\boldsymbol \bSigma$ evaluated at $\widehat{\bbeta}$: $\bSigma_n:=-\ddot{\ell}(\widehat{\bbeta})=
\frac{1}{n} \sum_{i=1}^{n}\bx_i\bx_i^T\ddot{b}(\bx_i^T\widehat{\bbeta}).$

To derive the asymptotic normality of the debiased estimator, first we need to provide the oracle inequality of Lasso for sub-Gaussion regression models. The oracle inequality is the error bound which connects the performance of an obtained estimator with the true parameter. It guarantees the consistency of the sparse high-dimensional estimator. The sharp $\ell_1$-error or $\ell_2$-error oracle inequalities are powerful non-asymptotical error bounds that control the error term in deriving the asymptotic normality of debiased Lasso. In \cite{bickel2009simultaneous}, the restricted eigenvalue condition for sample covariance matrix has been proposed to derive oracle inequalities. Here, we shall use an extended version presented in \cite{dedieu2019an} as a modified version.
\begin{definition}[A new RE conditions]\label{new_REC}
For a fixed matrix $\bSigma$ and constant $\gamma_{1}, \gamma_{2}>0 .$ The restricted eigenvalue condition $\mathrm{RE}(k, \gamma, \bSigma)$ holds if there exists $\kappa\left(s, \gamma_{1}, \gamma_{2}\right)$ satisfies
$$
0<\kappa\left(k, \gamma_{1}, \gamma_{2}\right) \leq \inf _{|S| \leq s} \inf _{\bm{z} \in \Lambda\left(S, \gamma_{1}, \gamma_{2}\right)} \frac{\bm{z}^{T} \bSigma \bm{z}}{n\|\bm{z}\|_{2}^{2}},
$$
where $\gamma=\left(\gamma_{1}, \gamma_{2}\right)$ and for every subset $S \subset\{1, \ldots, p\},$ the cone $\Lambda\left(S, \gamma_{1}, \gamma_{2}\right) \subset \mathbb{R}^{p}$ is defined as:
$\Lambda\left(S, \gamma_{1}, \gamma_{2}\right)=\left\{\bm{z} \in \mathbb{R}^{p}:\left\|\bm{z}_{S^{c}}\right\|_{1} \leq \gamma_{1}\left\|\bm{z}_{S}\right\|_{1}+\gamma_{2}\left\|z_{S}\right\|_{2}\right\}$.
\end{definition}

Now, we give the regularity conditions used for deriving the consistency of Lasso and asymptotic normality of the debiased estimator for sub-Gaussian regression models.
\begin{itemize}
\item[{(C1)}]
The covariates $\{\bx_i\}_{i = 1}^n$ are i.i.d. random vectors and $\max_{i \in[n]} \left\|x_{i}\right\|_{\infty} \leq L$.

\item[{(C2)}] There exists $B>0$ so that
$\|\bbeta^0\|_1\le B.$

 \item[{(C3)}]
(RE conditions: $\mathrm{RE}(s_0, \gamma, \mE(\bx_1 \bx_1^{T}))$) There exists $\kappa^*>0$, so that
\begin{center}
$\inf _{|S| \leq s_0} \inf _{\bm{z} \in \Lambda\left(S, \gamma_{1}, \gamma_{2}\right)} \frac{\bm{z}^{T} \mE(\bx_1 \bx_1^{T}) \bm{z}}{\|\bm{z}\|_{2}^{2}}>\kappa^*.$
\end{center}
\item[{(C4)}]
{(i)}
$\{y_i - \mathbb E (y_i | \bx_i)\}_{i = 1}^n$ are independent and non-degenerated $\{{\rm{subG}}(\sigma_i^2(\bx_i^T \bbeta^0))\}_{i = 1}^n$ distributed with $0<C_b^2 := \max_{1 \leq i \leq n} \sigma_i^2(\bx_i^T \bbeta^0)  < \infty$. Assume $\mathbb E (y_i | \bx_i):=\dot b(\bx_i^T \bbeta^0)$ and $0<\mathbb Var (y_i | \bx_i):=\ddot b(\bx_i^T \bbeta^0)< \infty$ for a function $b(\cdot)$ under {(C1,C2)}.

{(ii)}
The function $\ddot{b}(\theta)$ is Lipschitz continuous in a compact set,i.e. there exists $R >0$ so that
$ \underset{|a-\tilde{a}|\leq R}{\sup }
\frac{|\ddot{b}(a)-\ddot{b}(\tilde{a})|}{|a-\tilde{a}|} \leq C_R .$
\item[{(C5)}]{(i)} $n \ll p$ and for some $0 < r < 1/5$: $p=o(e^{n^{r}})$; {(ii)}  $s_0=o(\frac{n^{1 / 4}}{\log p}).$
\item[{(C6)}] Assume $\underset{n\rightarrow \infty}{\liminf} \left\|\bSigma^{-1}\right\|_{\mathrm{op}, 1}>0.$ Denote $r_{j}:=\sum_{i=1}^{p} \mathbbm{1}\{\bSigma^{-1}_{i j} \neq 0\}
 ,j\in [p]$, then we assume
\begin{center}
$\left\|\bSigma^{-1}\right\|_{\mathrm{op}, 1}^{2}\max _{j\in[p]} r_{j}= O(\sqrt{\log p}).$
\end{center}
\end{itemize}

\begin{remark}
Assumptions {(C1)} and {(C2)} are commonly used for GLMs \citep{blazere2014oracle,
ma2020global,zhang2022elastic}.
The restricted eigenvalue condition {(C3)} was first introduced by \cite{bunea2008honest} and \cite{bickel2009simultaneous} when proving consistency of Lasso.
Condition {(C4)} gives some regularities on the variance of th response $y_i$. {(C4)(i)} requires $y_i$ follows sub-Gaussian distribution with uniformly bounded variance. That we assume sub-Gaussianity is more general than the exponential family with compact parameter space assumption in ordinary GLMs, see Proposition 3.3(d) in \cite{zhang2020concentration}. {(C4)(ii)} is also used in the aforementioned articles on GLMs which is easy to satisfy for many commonly used functions.
Condition {(C5)} provides the required rate of $n, p$ and $s_0$.
\cite{yuyi2018} used similar conditions on Heissian matrix as {(C6)} which is more restrictive than our assumption.
\end{remark}

Following the new techniques in \cite{dedieu2019an} for linear models, we develop new oracle inequalities serving for the debiased Lasso analysis, which is crucial and fundamental in the theory for FDR control below. Existing oracle inequalities for GLMs in the high-dimension such as Corollary 3 in \cite{negahban2009unified}, Theorem III.8 in \cite{blazere2014oracle} and Theorem 8.1 in \cite{zhang2020concentration} are of order $O({s_0}\sqrt {\frac{{\log p}}{n}} )$ which is not sharper than our result. The following $\ell_1$-error and $\ell_2$-error oracle inequalities of Lasso for high-dimensional sub-Gaussian models matches the optimal minimax lower bounds $O({s_0}\sqrt {\frac{{\log (p/{s_0})}}{n}} )$ in \cite{ye2010rate}.

\begin{theorem} \label{lassobeta1}
Assume that the i.i.d. samples $\{(\bx_i, y_i)\}_{i=1}^{n}$ are governed by \emph{Sub-Gaussian GLMs} \eqref{subREG} satisfying (C4)(i). Let $\lambda = 12 \alpha \sigma L \sqrt{\frac{\log (2p e /s_0)}{n} \log \left( \frac{1}{\delta}\right)}$
where $0<{\delta}<1$ and $\alpha>1$ such that ${1 - \sqrt {\frac{{\log (2p)}}{{{\alpha ^2}\log (2pe/{s_0})}}} }>0$ Set $K_{\alpha}=\frac{{1 +\sqrt {\frac{{\log (2p)}}{{{\alpha ^2}\log (2pe/{s_0})}}} }}{{1 - \sqrt {\frac{{\log (2p)}}{{{\alpha ^2}\log (2pe/{s_0})}}} }}$. Suppose conditions {(C1), (C2)} and the RE condition (C3) with $\inf _{|S| \leq {s_0}} \inf _{z \in \Lambda\left(S, \gamma_{1}, \gamma_{2}\right)} \frac{z^{T} \mE(\bx_1 \bx_1^{T}) z}{\|z\|_{2}^{2}}>\kappa^*>0$ for $\gamma_{1}=\frac{\alpha}{\alpha-1}$ and $\gamma_{2}=\frac{\sqrt{s_0}}{\alpha-1}$. Then, with probability at least $1 - \delta$ we have
\begin{align*}
\Vert \widehat{\bbeta}-\bbeta^0\Vert_{2}\le \frac{2 \sqrt{s_0}}{\kappa^* \mathop {\inf }\limits_{\left| t \right| \le L(B + K_\alpha)}\ddot b(t)}\lambda=\frac{24 \alpha \sigma L \sqrt{s_0}}{\kappa^* \mathop {\inf }\limits_{\left| t \right| \le L(B + K_\alpha)}\ddot b(t)} \sqrt{\frac{\log (2p e /s_0)}{n} \log \left( \frac{1}{\delta}\right)},\\
\Vert \widehat{\bbeta}-\bbeta^0\Vert_{1}\le \frac{2{s_0}}{\kappa^* \mathop {\inf }\limits_{\left| t \right| \le L(B + K_\alpha)}\ddot b(t)}\lambda=\frac{24 \alpha \sigma L {s_0}}{\kappa^* \mathop {\inf }\limits_{\left| t \right| \le L(B + K_\alpha)}\ddot b(t)} \sqrt{\frac{\log (2p e /s_0)}{n} \log \left( \frac{1}{\delta}\right)}.
\end{align*}
\end{theorem}

Note that here ${\kappa^* \mathop {\inf }\limits_{\left| t \right| \le L(B + K_\alpha)}\ddot b(t)}>0$ by the variance assumption. This is because the non-degenerate exponential distribution family will not concentrate at one point under \eqref{glmlike}. Based on Theorem \ref{lassobeta1}, the asymptotic normality for all linear combinations of debiased Lasso estimator is established in Proposition \ref{main-thm1} below.
\begin{proposition}\label{main-thm1}
Suppose {(C1)--(C6)} hold. The estimator $\widehat{\bbeta}$ is calculated as $\eqref{lassoes}$ with the tuning parameter $\lambda \asymp \sqrt{\frac{\log (p/s_0)}{n}}$. Set $\lambda_n \asymp \left\|\bSigma^{-1}\right\|_{\mathrm{op}, 1}s_0\sqrt{\frac{\log p}{n}}$ in \eqref{Theta} to derive $\widehat{\bTheta}$. Then for the debiased estimator  $\widehat{\bbeta}^d$ defined in $\eqref{debia}$ with $\widehat{\bbeta}$ and $\widehat{\bSigma}$, and any vector $\boldsymbol{c}\in \mathbb{R}^p$ satisfying $\|\boldsymbol{c}\|_1 = 1$ and $\boldsymbol{c}^T \bSigma^{-1}\boldsymbol{c} \rightarrow \sigma^2 \in (0,\infty)$, when $n\to \infty$, we have
\begin{center}
$\sqrt{n}\boldsymbol{c}^T (\widehat{\bbeta}^d - \bbeta^0) \stackrel{d}{\to} \mathcal{N}(0, \sigma^2)$
and $\frac{\sqrt{n}\boldsymbol{c}^T (\widehat{\bbeta}^d - \bbeta^0)}
{\sqrt{\boldsymbol{c}^T
\widehat{\bTheta}\boldsymbol{c}}}\stackrel{d}{\to} \mathcal{N}(0, 1).$
\end{center}
\end{proposition}
According to the asymptotic normality in Proposition \ref{main-thm1}, we can use the debiased estimator $\widehat{\bbeta}^d$ for constructing the test statistics in the succeeding section.

\section{Multiple Testing Procedures}
In this section, we discuss procedures for large-scale multiple testing. A data-based multiple testing procedure that can be applied to model selection under high-dimensional conditions is provided based on the debiased Lasso estimator. Further, we prove that it can approximately control directional FDR at a pre-specified level $\alpha$.

Furthermore, we can also control the number of directional falsely discovered variables (FDV) and the directional family-wise error rate (FWER). We also study statistical power and related theoretical results. A testing procedure for the two-sample multiple hypothesis testing problem is also given.

\subsection{Problem formulation}
For high-dimensional regression models defined in Section \ref{subsec:model}, we consider the following multiple testing problem:
\begin{equation}\label{H0H1}
H_{0,j}:\beta_j=0 ~\text{vs}~H_{0,j}:\beta_j\not=0, j\in[p].
\end{equation}
The $p$ hypotheses here are tested at the same time and this problem is actually equivalent to select the set $\widehat{S}$ consisting of significant features. In many studies, researchers are concerned about whether $\widehat{S}$ is consistent with the true set and the related standards include FDR, FDV, FWER. However, in this article we care more about the type-S error \citep[S stands for signs;][]{gelman2000type}, which is telling the sign of the paramter of interest wrongly. Specifically, for the test \eqref{H0H1} considered in this section about, it is a type-S error when we make a claim of $\beta_j<0$ while the fact is $\beta_j>0$. In fact, the connotations of $\beta_j>0$ and $\beta_j<0$ in practical applications can be very different. In many situations, they make more sense than only questioning whether the effect is nonzero. For example, for genes, it is meaningless to judge whether a gene has an effect on a certain disease. What we really care about is whether it has a positive effect or a negative effect on it. Therefore, we need some criteria that can describe this type-S error and a corresponding algorithm to meet these standards.

Following the definitions given by \cite{javanmard2019false}, we state the definition of directional FDR and further define directional FDV and directional FWER which characterize the type-S error for our testing problem \eqref{H0H1}.
Suppose the true value of $\bbeta$ in \eqref{score} is $\bbeta^0$.
Let $S=\operatorname{supp}{}
\left(\boldsymbol{\beta}^{0}\right)=\{j \in[p]:\left.\beta_{j}^{0} \neq 0\right\} \in [p]$. The set $\widehat{S}$ is the selected subset of feature variables with $\widehat{\text{sign}}_j\in\{1,-1\}, j \in \widehat{S}$ as the estimation of the sign of $\beta^0_j$.

\begin{definition}\label{fdrdir}
The directional false discovery proportion (FDP), denoted by $\textup{FDP}_\textup{d}$, and the directional false discovery rate (FDR), denoted by $\textup{FDR}_\textup{d}$, are defined as
\begin{equation}\nonumber
\textup{FDP}_\textup{d}=  \frac{|\{j\in \widehat{S}:\, \widehat{\sign}_j \neq \sign(\beta^0_j)\}|}{|\widehat{S}|\vee 1},~~~\textup{FDR}_\textup{d}= \mE[ \textup{FDP}_\textup{d}].
\end{equation}
\end{definition}

\begin{definition}\label{fdvdir}
The directional falsely discovered variables (FDV), denoted by $\textup{FDV}_\textup{d}$,  is
$\textup{FDV}_\textup{d}= \mE\left[ \sum_{j\in \widehat{S}} \mathbbm{1}
\{\widehat{\sign}_j \neq \sign(\beta^0_j)\} \right]. $ The directional family-wise error rate (FWER), denoted by $\textup{FWER}_\textup{d}$, is
$\textup{FWER}_\textup{d}= \mP( \sum_{j\in \widehat{S}} \mathbbm{1}
\{\widehat{\sign}_j \neq \sign(\beta^0_j)\} \geq 1).$
\end{definition}

Likewise, we give the definition of the directional statistical power, which is the proportion of the successfully identified relevant variables .

\begin{definition}\label{powerdir}
The directional statistical power is
$\textup{Power}_\textup{d}= \mE[\frac{|\{j\in \widehat{S}:\, \widehat{\sign}_j = \sign(\beta^0_j)\}|}{|S|\vee 1}].$
\end{definition}

\subsection{The procedures for directional FDR and FDV control}
In this section, we give the high-dimensional GLM Multiple Testing procedure (GMT) that controls directional FDR, and the procedure that controls directional FDV and FWER, denoted by ${\text{GMT}_{\text{FDV}}} $. To make the testing procedure more desirable and powerful, the theory of directional FDR and FDV control need an estimator that is tractable with limit distributions so that the inference procedure can be drawn. The debiased estimator and Proposition \ref{main-thm1} are employed to take on this task.\\

\textbf{Directional FDR control}. Using $\widehat{\boldsymbol{\beta}}^d$ defined in \eqref{debia}, we give the standardized statistic
\begin{equation}\label{standard1}
T_j = \frac{\sqrt{n} \widehat{\beta}^d_j}{\sqrt{\widehat{\Theta}_{jj}}}, j \in [p].
\end{equation}
Given threshold $t>0$, if $|T_j| \geq  t$, we reject $H_{0,j}: \beta_j=0$ if $|T_j|$ is large enough.

To control the directional FDR at the pre-specified level $0<\alpha<1$, the threshold naturally need to satisfy the condition
\begin{equation}\label{t-first}
\frac{\sum_{j\in S_{\geq0}} \mathbbm{1}( T_{j} \leq -t)
+\sum_{j\in S_{\leq0}} \mathbbm{1}( T_{j} \geq t)}
{ \big\{\sum_{j=1}^p \mathbbm{1}\{ |T_{j}| \ge t\} \vee 1 \big\}} \le \alpha,
\end{equation}
where $S_{\leq0}\equiv\{j\in[p]:\beta^0_j\leq 0\}$, and $S_{\geq 0}\equiv\{j\in[p]:\beta^0_j\geq 0\}$. In reality, however, both $S_{\geq0}$ and $S_{\leq0}$ are unknown. Thus, it is not easy to find a useable threshold $t$. According to the form of the numerator of \eqref{t-first}, we denote $\widetilde{G}(t)=\frac{1}{p} \left\{ \sum_{j\in S_{\geq0}} \mathbbm{1}( T_{j} \leq -t)
+\sum_{j\in S_{\leq0}} \mathbbm{1}( T_{j} \geq t)\right\}.$

Thanks to the asymptotic normality of $T_j$, when the sample size is large enough, we can use the tail of the standard normal distribution $G(t)=2-2\Phi(t)$ to approximate $\widetilde{G}(t)$, where $\Phi(t)$ is the distribution function of the standard normal distribution. In Proposition 5 of Supplementary Material, we shall prove that $G(t)$ is a good approximate for $\widetilde{G}(t)$.

Now, we are ready to give the GLM Multiple Testing (GMT) procedure for directional FDR control. In beloq, we will formally address the result that it is able to control directional FDP and FDR at the target level and also achieve high statistical power in the asymptotic sense.

For the pre-specified target $\textup{FDR}_\textup{d}$ level $\alpha \in [0,1]$ and $t_p=\sqrt{2\log p-2\log\log p}$, the GMT procedure takes the standardized test statistic $T_j, j \in [p]$ as input and process it with the following steps.\\

\textbf{GMT Procedure.} Calculate the threshold
\begin{equation}\label{t-fdr}
t_0 =\inf\bigg\{ 0\le t\le t_p: \frac{pG(t)}{ \sum_{j=1}^p \mathbbm{1}\{ |T_{j}| \ge t\} \vee 1   } \le \alpha  \bigg\}.
\end{equation}
If $t_0$ in \eqref{t-fdr} does not exist,  we set $t_0=\sqrt{2\log p}.$
For each $j\in[p]$, if $|T_j|\geq t_0$, then reject $H_{0,j}$. The sign of $T_j$ is taken as the estimator of $\sign(\beta^0_j)$.\\

\textbf{Directional FDV control}. In some cases, people are more interested in controling the specific number of FDV instead of controlling the inexact FDR. Since controlling $\textup{FDV}_\textup{d}$ under level $u (0<u<1)$ is directly related to controlling FWER by definition, we provide the approaches to controlling FDV and controlling FWER together, denoted by ${\text{GMT}_{\text{FDV}}}$.\\

\textbf{$\text{GMT}_{\textup{FDV}}$ Procedure.} For $\textup{FDV}_\textup{d}$ control, the tolerable number of $\textup{FDV}_\textup{d}$ is some integer $u<p$. For $\textup{FWER}_\textup{d}$ control, pre-specified a level $u \in (0,1)$. Denote $t_{\textup{FDV}}=G^{-1}\left(\frac{u}{p} \right).$ For each $j\in[p]$, if $|T_j|\geq t_{\textup{FDV}}$, then we reject $H_{0,j}$. The sign of $T_j$ is taken as the estimator of $\sign(\beta^0_j)$.\\

\textbf{Theoretical results of multiple testing procedures}. In this part, we give theoretical results which demonstrate that our procedures can asymptotically control directional FDR and FDV. This aim is highly related to the sign consistency which ensures the estimated sign and the true sign are equal with high probability. The proposed test statistic used in our procedure is a function of the estimated sign, and hence controlling directional FDR and FDV can be seen as an extension of sign consistency in the multiple testing framework. \cite{zhao2006model} studied sign consistency of Lasso based on the Irrepresentable Condition. Also, \cite{jia2013lasso} studied sign consistency of sparse Poisson-like heteroscedasticity model. For high-dimensional linear models, \cite{javanmard2019false} proposed a multiple testing procedure which is proved to control the directional FDR below a pre-assigned significance level. Our theories in this section generalizes it to high-dimensional sub-Gaussian regression models.

Besides the assumptions used for Proposition \ref{main-thm1}, we need some additional assumptions to limit $\bTheta:=\bSigma^{-1} = (\bTheta_1, \dots, \bTheta_p)$. Define the standardized matrix $\bTheta^0$ whose $jk$-th element is
\begin{equation}\label{Omega-0}
\Theta_{jk}^0=\frac{\Theta_{jk}}{\sqrt{\Theta_{jj}\Theta_{kk}}}, 1\leq j,k\leq p.
\end{equation}
Given any $\gamma>1/2$, denote $\mathcal{A}(\gamma) =\{(j,k): 1\leq j,k\leq p, |\Theta_{jk}^0|\geq(\log p)^{-2-\gamma}\}, \mathcal{B}(\rho) = \{(j,k): 1\leq j,k\leq p, |\Theta_{jk}^0|>\frac{1-\rho}{1+\rho}\}.$ The following conditions {(C7)} and {(C8)} are similar to the conditions used for Theorem 3.1 in \cite{javanmard2019false}.
\begin{itemize}
\item[{(C7)}] There exist constants $\gamma>1/2$ and $\rho\in (0,1/3)$ so that
$|\mathcal{A}(\gamma)| =o(p^{1+\rho}),~
|\mathcal{B}(\rho)|=O(p) .$
\item[{(C8)}] There exist constants $c_1,c_2$ so that $0<c_1\leq \sigma_{\min}(\bSigma)\leq\sigma_{\max}(\bSigma) \leq c_2.$
\end{itemize}
Condition {(C8)} implies $0<1/c_2\leq \sigma_{\min}(\bTheta)\leq\sigma_{\max}(\bTheta) \leq 1/c_1$, and for all $j\in[p]$, $\Theta_{jj} \in [1/c_2,1/c_1].$ Now, we give the theoretical result of GMT procedure.
\begin{theorem}\label{main-thm2}
Under {(C1)--(C8)}, for GMT procedure, we have, for pre-specified target level $\alpha$ and any $\varepsilon>0$,
\begin{equation}\label{fdp}
\lim_{(n,p) \rightarrow \infty} \mP\left( \frac{\textup{FDP}_\textup{d}}{\alpha(1- \frac{s_0}{2p})}\leq 1+\varepsilon \right)=1,~~\limsup_{(n,p) \rightarrow \infty}  \frac{\textup{FDR}_\textup{d}}{\alpha(1- \frac{s_0}{2p})}\leq 1.
\end{equation}
\end{theorem}

Theorem \ref{main-thm2} reveals that our method not only controls directional FDR, but also controls directional FDP. In practical applications, we are more concerned about whether we can control FDP in the current implement, rather than the expectation FDR. Also, the variance of FDP may be very large in some settings \citep{owen2005variance}, and only controlling FDR is not enough for controlling the variance of it.

Next, we have similar results for $\text{GMT}_{\textup{FDV}}$ procedure on controlling directional FDV and directional FWER.
\begin{theorem}\label{main-thm3}
Under conditions in Proposition \ref{main-thm1} and {(C8)}, given the target $\textup{FDV}_\textup{d}$ level $u<p$, for $\text{GMT}_{{FDV}}$ procedure we get
\begin{equation}\label{fdv}
\limsup_{(n,p) \rightarrow \infty}  \frac{\textup{FDV}_\textup{d}}{u(1- \frac{s_0}{2p})}\leq 1,
\end{equation}
Moreover, if $0<u<1$, we have
 \begin{equation}\label{fwer}
\limsup_{(n,p) \rightarrow \infty}  \frac{\textup{FWER}_\textup{d}}
{u(1- \frac{s_0}{2p})}\leq 1.
\end{equation}
\end{theorem}

Below we also give the rate of the statistical power of GMT procedure.

\begin{theorem} \label{power-thm}
Assume the conditions of Theorem {\ref{main-thm2}} hold. Additionally, we assume for $j \in S=\operatorname{supp}
\left(\boldsymbol{\beta}^{0}\right)$ and $s_0=|S|$,
$|\beta_j^0|>\sqrt{\frac{8\Theta_{jj}\log(p/s_0)}{n}}.$ Let $Q(t)=G(t)/2$.
Then, for pre-specified target $\textup{FDR}_\textup{d}$ level $\alpha$, with GMT procedure, we have
\begin{equation*}
\liminf_{n\rightarrow \infty} \frac{\textup{Power}_\textup{d}}{\frac{1}{s_0}\sum_{j\in S}
Q\left[G^{-1}(\frac{\alpha s_0}{p})-\frac{\sqrt{n}|\beta_j^0|}{\sqrt{\Theta_{jj}}}\right]}   \geq   1.
\end{equation*}
\end{theorem}

\begin{corollary} \label{power-cor2}
Under conditions of Theorem \ref{power-thm}, if additionally we have, as $n,p \rightarrow  \infty $,
$\frac{\sqrt{n}\beta_{\min}}{\sqrt{\max_{j\in[p]}\Theta_{jj}}}
-\sqrt{2\log\left(\frac{2p}{\alpha s_0}\right)}
 \rightarrow  \infty.$ Then as $n,p \rightarrow  \infty $, we have $\textup{Power}_\textup{d} \rightarrow  1.$
\end{corollary}
Corollary \ref{power-cor2} states that if the signal-to-noise ratio $\frac{\beta_{\min}}{\sqrt{\max_{j\in[p]}\Theta_{jj}}} \gtrsim \sqrt{\frac{\log (p / s_0)}{n}}$, then the controlling procedure GMT is consistent in the aspect of power.

\subsection{Two-Sample Multiple Hypothesis Testing}
In some cases, we have two separate regression models of the same dimensions, and we are interested in testing the difference between their parameters.  In this section, we will give test procedures for these scenarios.

For $\ell =1,2$, we have mutually independent random design matrices $\bX^{(\ell)}=(\bx_{1}^{( \ell)}, \ldots, \bx_{n_{\ell}}^{(\ell)})^{T} \in \mathbb{R}^{n_{\ell} \times p}$ and response variable $\bY^{(\ell)}\in\mathbb{R}^{n_{\ell}}$, here $n_1\asymp n_2.$
For $i=1,\ldots,n_{\ell}$, similar to the definition in Section \ref{subsec:model}, assume $y_i^{\ell} - \mathbb E (y_i^{\ell}| \bx_i^{\ell}) \sim \mathrm{subG}(\sigma_i^2(({\bx_i^{(\ell)}})^T \bbeta^{\ell}))$ conditional on $\bx_i^{\ell}$ with variance $\ddot b(({\bx_i^{(\ell)}})^T \bbeta^{\ell})\le \sigma_i^2(({\bx_i^{(\ell)}})^T \bbeta^{\ell})$.

The tests of the regression coefficients of these two models are
\begin{equation*}
H_{0,j}:\beta_j^{(1)}=\beta_j^{(2)}\quad \text{vs} \quad H_{1,j}:\beta_j^{(1)}\neq \beta_j^{(2)},j\in[p].
\end{equation*}
Denote $\tilde{S} = \{j \in [p]:\left. \beta_{j}^{(1)} - \beta_{j}^{(2)} \neq 0\right\} \in [p]$. With a slight abuse of notation,  let $\widehat{S}$ again denote the selected subset of features in two-sample setting and $\widehat{\text{sign}}_j \in \{1,-1\}, j \in \widehat{S}$ denote the estimation of the sign of $(\beta_{j}^{(1)} - \beta_{j}^{(2)})$ in this section.

Under the two-sample settings, we give the definition of the directional FDR, directional FDV and directional FWER.

\begin{definition}\label{twosamplefdrdir}
Under the two-sample settings, the directional FDP/FDR, directional FDV and directional FWER are defined as
\begin{center}
$\textup{FDP}_{d}= \frac{|\{j\in \widehat{S}:\, \widehat{\sign}_j \neq \sign(\beta_{j}^{(1)} - \beta_{j}^{(2)})\}|}{|\widehat{S}|\vee 1},~\textup{FDR}_\textup{d}=\mE[ \textup{FDP}^d],$ and
\end{center}
$\textup{FDV}_\textup{d}=\mE[ \sum_{j\in \widehat{S}} \mathbbm{1}
\{\widehat{\sign}_j \neq \sign(\beta_{j}^{(1)} - \beta_{j}^{(2)})\} ]$, $\textup{FWER}_\textup{d}=\mP( \sum_{j\in \widehat{S}} \mathbbm{1}\{\widehat{\sign}_j \neq \sign(\beta_{j}^{(1)} - \beta_{j}^{(2)})\} \geq 1 ).$
\end{definition}

For each model, we first use the debiased estimator defined in \eqref{debia}, and, like \eqref{Theta}, we gives the estimation of $ \bTheta^{(\ell)}=(\bSigma^{(\ell)})^{-1}$, denoted by $\widehat{\bTheta}{}^{(\ell)}.$ Similar to \eqref{standard1}, we give the two sample test statistic (as the $t$-test form): $M_j =\frac{\widehat{\beta}^{(1)}_j-\widehat{\beta}^{(2)}_j}
{\sqrt{\widehat{\Theta}_{jj}^{(1)}/n_1+\widehat{\Theta}_{jj}^{(2)}/n_2}}.$

For the pre-specified target $\textup{FDR}_\textup{d}$ level $\alpha \in [0,1]$ and $t_p=\sqrt{2\log p-2\log\log p}$, we give the two-sample multiple hypothesis test procedure, named GMT2, that controls directional FDR.\\

\textbf{GMT2 Procedure.} Calculate the threshold
\begin{equation}\label{t-fdr2}
t_0 =\inf\bigg\{ 0\le t\le t_p: \frac{pG(t)}{ \sum_{j=1}^p \mathbbm{1}\{ |M_{j}| \ge t\} \vee 1   } \le \alpha  \bigg\}.
\end{equation}
If $t_0$ in \eqref{t-fdr2} does not exist, then set $t_0=\sqrt{2\log p}.$
For each $j\in[p]$, if $|M_j|\geq t_0$, then reject $H_{0,j}$. The sign of $M_j$ is taken as the estimator of $\sign(\beta_{j}^{(1)} - \beta_{j}^{(2)})$.

Analogous to $\text{GMT}_{\textup{FDV}}$ procedure, we also give the two-sample multiple hypothesis test process, named $\text{GMT2}_\textup{FDV}$ that controls FDV and FWER.\\

\textbf{$\text{GMT2}_{\textup{FDV}}$ Procedure.} For $\textup{FDV}_\textup{d}$ control, the tolerable number of $\textup{FDV}_\textup{d}$ is $u<p$. For $\textup{FWER}_\textup{d}$ control, the pre-specified target level is $0<u<1$. Denote $t_{\textup{FDV}}=G^{-1}\left(\frac{u}{p} \right).$ For each $j\in[p]$, if $|M_j|\geq t_{\textup{FDV}}$, then reject $H_{0,j}$. The sign of $M_j$ is taken as the estimator of $\sign(\beta^{(1)}_j-\beta^{(2)}_j)$.\\

Similar to the definition of $\bTheta$ in \eqref{Omega-0}, we define $\widetilde{\bTheta}=(\widetilde{\Theta}_{jk})_{1\le j, k\le p}$ by $\widetilde{\bTheta}=\frac{\bTheta^{(1)}/n_1+\bTheta^{(2)}/n_2}{1/n_1+1/n_2}.$ And the standardized matrix $\widetilde{\bTheta}{}^0=(\widetilde{\Theta}_{jk}^0)_{1\le j, k\le p}$ is defined as $\widetilde{\Theta}_{jk}^0 = \frac{\widetilde{\Theta}_{jk}}{\sqrt{\widetilde{\Theta}_{jj}\widetilde{\Theta}_{kk}}}.$ For any given constant $\gamma>1/2$, denote $\widetilde{\mathcal{A}}(\gamma) = \{(j,k): 1\leq j,k\leq p, |\widetilde{\Theta}_{jk}^0|\geq(\log p)^{-2-\gamma}\}$ and $\widetilde{\mathcal{B}}(\rho) =\{(j,k): 1\leq j,k\leq p, |\widetilde{\Theta}_{jk}^0|>\frac{1-\rho}{1+\rho}\}.$ The following condition {(C9)} is similar to condition {(C7)}.
\begin{itemize}
\item[{(C9)}] There exists constants $\gamma>1/2$ and $\rho\in (0,1/3)$ so that
$|\widetilde{\mathcal{A}}(\gamma)| =o(p^{1+\rho}), \quad
|\widetilde{\mathcal{B}}(\rho)|=O(p) .$
\end{itemize}

Similar to Remark under {(C8)}, when {(C8)} holds for both $\bSigma^{(1)}$ and $\bSigma^{(2)}$, then $0<1/c_2\leq\sigma_{\min}(\widetilde{\bTheta})\leq\sigma_{\max}(\widetilde{\bTheta}) \leq 1/c_1$, and for any $j\in[p]$, $\widetilde{\Theta}_{jj} \in [1/c_2,1/c_1].$ The main theorems of the two-sample test are given below.

\begin{theorem}\label{twosample-thm1}
Suppose the conditions of Proposition \ref{main-thm1} and {(C8)}, {(C9)} hold. For GMT2 procedure with pre-specified target level $\alpha$, we have, for any $\varepsilon>0$, $\lim_{(n,p) \rightarrow \infty} \mP( \frac{\textup{FDP}_\textup{d}}{\alpha(1- \frac{s_0}{2p})}\leq 1+\varepsilon )=1$. Further we have $\lim_{(n,p) \rightarrow \infty}  \frac{\textup{FDR}_\textup{d}}{\alpha(1- \frac{s_0}{2p})}\leq 1.$
\end{theorem}

\begin{theorem}\label{twosample-thm2}
Assume that the conditions of Proposition \ref{main-thm1} and {(C8)} holds for both models in the two-sample settings. Given the target $\textup{FDV}_\textup{d}$ level $u<p$, for $\text{GMT2}_{{FDV}}$ procedure we get $\lim_{(n,p) \rightarrow \infty}  \frac{\textup{FDV}_\textup{d}}{u(1- \frac{s_0}{2p})}\leq 1$.  Further for $0<u<1$ we have $\lim_{(n,p) \rightarrow \infty}  \frac{\textup{FWER}_\textup{d}}{u(1- \frac{s_0}{2p})}\leq 1.	$
\end{theorem}

Like the ordinary multiple testing setting (i.e., the setting discussed in the last section), Theorem \ref{twosample-thm1} and Theorem \ref{twosample-thm2} indicate that GMT2 and $\text{GMT2}_{\textup{FDV}}$ can asymptotically control directional FDR/FDP and FDV/FWER respectively under two-sample settings.

\section{Numerical Experiments}

In order to examine the performance of the proposed methods, we conduct experiments using GMT in the simulations. Under two commonly used GLMs --- logistic regression and Poisson regression, the principal evaluation criteria are directional FDR and statistical power.

\subsection{Logistic Regression}

Consider the logistic regression model, where each row $\bx_i$ of the random design matrix $\boldsymbol{X} \in \mathbb{R}^{n \times p}$ is i.i.d. taken from a uniform distribution $\rm{U}(- 1, 1)$.  The number of non-zero elements of the true regression coefficient $\boldsymbol{\beta}^0 \in \mathbb{R}^{p}$ is $s_0$. Given the \emph{signal magnitude} $A>0$,
set the first $s_0$ dimensions of $\boldsymbol{\beta}^0$ are selected from $\{\pm A\}$, with half positive and half negative values, and the last $p - s_0$ dimensions are taken as 0, i.e.
${\beta _0} = (\underbrace { \pm A, \cdots , \pm A}_{{s_0}},0, \cdots ,0).$

The response variable $y_i, i \in [n]$ comes from Bernoulli distributions with the mean value $\frac{\exp (\bx_i^T \boldsymbol{\beta}^0)}{1 + \exp (\bx_i^T \boldsymbol{ \beta}^0)}.$ We compare 3 methods: 1 The GMT method proposed in this article; 2. The LMT method is in \cite{ma2020global}; 3. The knockoff method is in \cite{barber2015controlling}.

GMT and LMT methods involve Lasso estimation. For the penalty coefficient $\lambda$ used in the Lasso method, we use a 5-fold cross-validation to select $\lambda$ such that the error obtained in cross-validation is the smallest. For GMT method, the precision matrix needs to be estimated. We use CLIME with the optimal tuning parameter $\lambda_n$ selected by  2-fold cross-validation method. We select $p = 600$, $n \in \{400, 600, 800, 1000, 1200\}$, sparsity level $s_0 = 50$, non-zero signal magnitude $A = 0.5$, and the target directional FDR level $\alpha = 0.2$.
The experiment is repeated 100 times under all settings.

The results are shown in Figure \ref{logitfdrpower}. The empirical directional FDR of the three methods do not exceed the pre-specified level 0.2. The GMT method in this article has similar performance with the LMT method. When $n$ is large, the statistical power of GMT method is relatively higher than that of the knockoff method.

\begin{figure}[htbp]
    \centering
    \includegraphics[width = 1.05\textwidth]{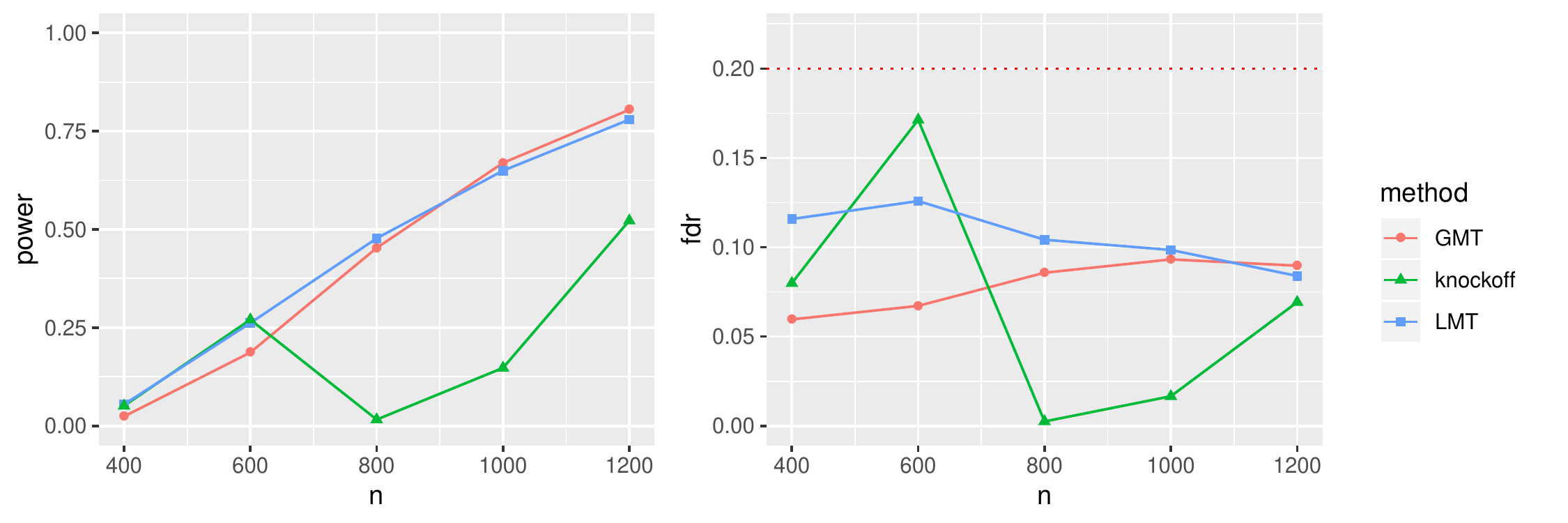}
    \caption{The empirical directional FDR and power of logistic regression. We set $p = 600$, $s_0 = 50$, non-zero signal strength $A = 0.5$. The number of trials is 100.}
    \label{logitfdrpower}
\end{figure}

\subsection{Poisson Regression}

For Poisson regression, each row of the random design matrix $\boldsymbol{X} \in \mathbb{R}^{n \times p}$ independently comes from a uniform distribution $\rm{U}(-0.6, 0.6)$. The distribution of the true regression coefficient $\boldsymbol{\beta}_0$ is the same as the logistic regression part, with the non-zero signal strength $A = 0.2$. The $y_i$ comes from a conditional Poisson distribution with the conditional mean $\exp (x_i^T \boldsymbol{\beta }_0)$.

As LMT procedure is only proposed for logistic regression, in the part we only compare GMT and knockoff methods.
Since the link function of Poisson regression is an exponential function with a high degree of non-linearity, a relatively large sample size is required to obtain a relatively ideal result for both methods. Therefore, in the figure \ref{poissonfdrpower} we show the results of $n = 1000$ and $n = 2000$ respectively with $p = 300$ and non-zero signal strength $A = 0.2$. The empirical directional FDR of the two methods is controlled below 0.2, and the GMT method is relatively more effective than the knockoff method in the aspect of the statistical power.

\begin{figure}[htbp]
    \centering
    \includegraphics[width = 1.05\textwidth]{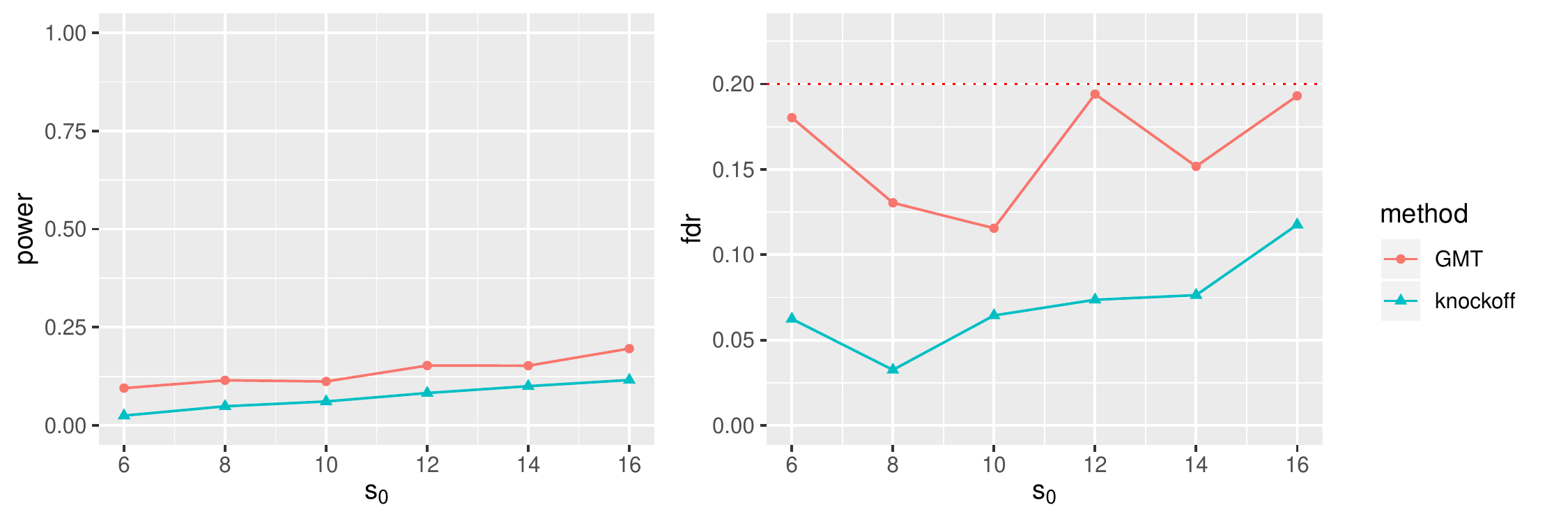}
    \includegraphics[width = 1.05\textwidth]{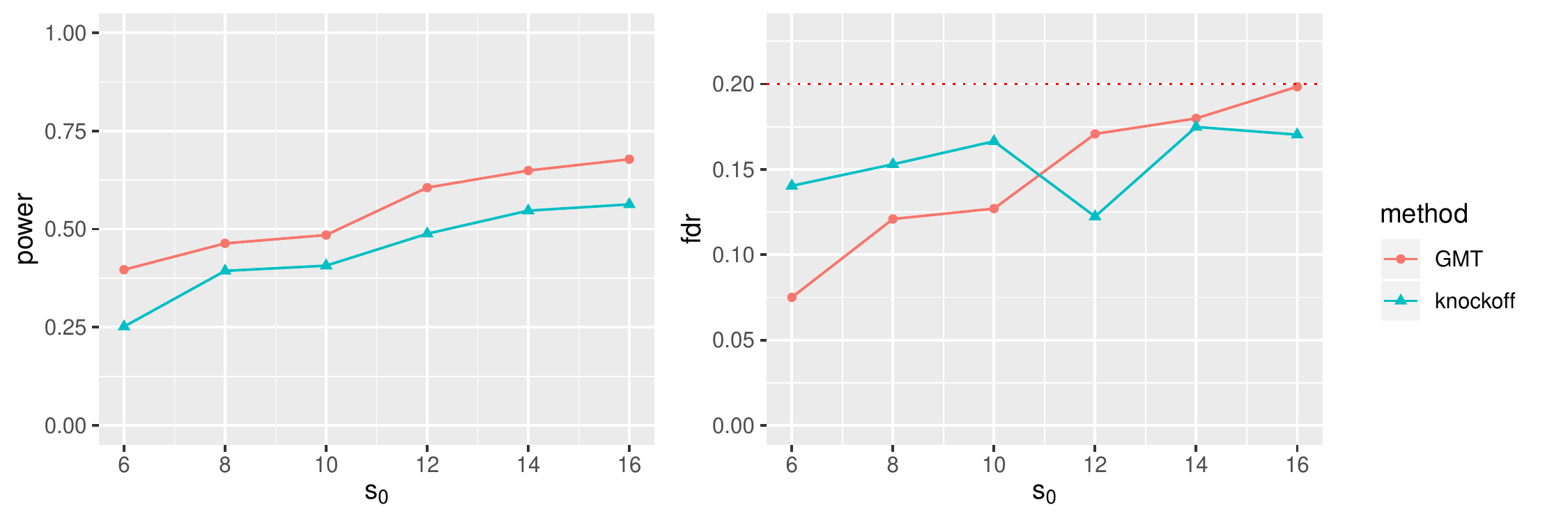}
    \caption{The empirical directional FDR and power of Poisson regression numerical experiments. The upper two plots are the results for $n = 1000$ and the lower two are for $n = 2000$. The number of trials is 100.}
    \label{poissonfdrpower}
\end{figure}


\section*{Acknowledgement}
The authors are in alphabetical order contributed equally to this work. This work is supported in part by the University of Macau under UM Macao Talent Programme (UMMTP-2020-01).

\bibliographystyle{apalike}
\bibliography{Directional_FDR}

\begin{thebibliography}{}

\bibitem[Barber and Cand{\`e}s, 2015]{barber2015controlling}
Barber, R.~F. and Cand{\`e}s, E.~J. (2015).
\newblock Controlling the false discovery rate via knockoffs.
\newblock {\em The Annals of Statistics}, 43(5):2055--2085.

\bibitem[Benjamini and Hochberg, 1995]{BH95}
Benjamini, Y. and Hochberg, Y. (1995).
\newblock Controlling the false discovery rate: a practical and powerful
  approach to multiple testing.
\newblock {\em Journal of the Royal statistical society: series B},
  57(1):289--300.

\bibitem[Benjamini and Yekutieli, 2001]{benjamini2001control}
Benjamini, Y. and Yekutieli, D. (2001).
\newblock The control of the false discovery rate in multiple testing under
  dependency.
\newblock {\em Annals of statistics}, pages 1165--1188.

\bibitem[Bickel et~al., 2009]{bickel2009simultaneous}
Bickel, P.~J., Ritov, Y., and Tsybakov, A.~B. (2009).
\newblock Simultaneous analysis of lasso and dantzig selector.
\newblock {\em The Annals of Statistics}, 37(4):1705--1732.

\bibitem[Blazere et~al., 2014]{blazere2014oracle}
Blazere, M., Loubes, J.-M., and Gamboa, F. (2014).
\newblock Oracle inequalities for a group lasso procedure applied to
  generalized linear models in high dimension.
\newblock {\em IEEE Transactions on Information Theory}, 60(4):2303--2318.

\bibitem[Bunea, 2008]{bunea2008honest}
Bunea, F. (2008).
\newblock Honest variable selection in linear and logistic regression models
  via $\ell_1$ and $\ell_1+\ell_2$ penalization.
\newblock {\em Electronic Journal of Statistics}, 2:1153--1194.

\bibitem[Cai et~al., 2011]{cai2011constrained}
Cai, T.~T., Liu, W., and Luo, X. (2011).
\newblock A constrained $\ell_1$ minimization approach to sparse precision
  matrix estimation.
\newblock {\em Journal of the American Statistical Association},
  106(494):594--607.

\bibitem[Cai et~al., 2016]{cai2016}
Cai, T.~T., Liu, W., and Zhou, H.~H. (2016).
\newblock Estimating sparse precision matrix: Optimal rates of convergence and
  adaptive estimation.
\newblock {\em The Annals of Statistics}, 44(2):455--488.

\bibitem[Candes et~al., 2018]{candes2018panning}
Candes, E., Fan, Y., Janson, L., and Lv, J. (2018).
\newblock Panning for gold : 'model-x' knockoffs for high dimensional
  controlled variable selection.
\newblock {\em Journal of the Royal Statistical Society: Series B},
  80(3):551--577.

\bibitem[{Dedieu}, 2019]{dedieu2019an}
{Dedieu}, A. (2019).
\newblock An error bound for lasso and group lasso in high dimensions.
\newblock {\em arXiv preprint arXiv:1912.11398}.

\bibitem[Durrett, 2019]{durrett2019}
Durrett, R. (2019).
\newblock {\em Probability: theory and examples}, volume~49.
\newblock Cambridge university press.

\bibitem[Fan et~al., 2020a]{fan2020statistical}
Fan, J., Li, R., Zhang, C.-H., and Zou, H. (2020a).
\newblock {\em Statistical foundations of data science}.
\newblock CRC press.

\bibitem[Fan et~al., 2020b]{fan2020rank}
Fan, Y., Demirkaya, E., Li, G., and Lv, J. (2020b).
\newblock Rank: large-scale inference with graphical nonlinear knockoffs.
\newblock {\em Journal of the American Statistical Association},
  115(529):362--379.

\bibitem[Gelman and Tuerlinckx, 2000]{gelman2000type}
Gelman, A. and Tuerlinckx, F. (2000).
\newblock Type s error rates for classical and bayesian single and multiple
  comparison procedures.
\newblock {\em Computational Statistics}, 15(3):373--390.

\bibitem[Huang et~al., 2021]{huang2021weighted}
Huang, H., Gao, Y., Zhang, H., and Li, B. (2021).
\newblock Weighted lasso estimates for sparse logistic regression:
  non-asymptotic properties with measurement errors.
\newblock {\em Acta Mathematica Scientia}, 41(1):207--230.

\bibitem[Javanmard and Javadi, 2019]{javanmard2019false}
Javanmard, A. and Javadi, H. (2019).
\newblock False discovery rate control via debiased lasso.
\newblock {\em Electronic Journal of Statistics}, 13(1):1212--1253.

\bibitem[Javanmard and Montanari, 2014]{javanmard2014confidence}
Javanmard, A. and Montanari, A. (2014).
\newblock Confidence intervals and hypothesis testing for high-dimensional
  regression.
\newblock {\em The Journal of Machine Learning Research}, 15(1):2869--2909.

\bibitem[Jia et~al., 2013]{jia2013lasso}
Jia, J., Rohe, K., and Yu, B. (2013).
\newblock The lasso under poisson-like heteroscedasticity.
\newblock {\em Statistica Sinica}, pages 99--118.

\bibitem[Liu, 2013]{liu2013gaussian}
Liu, W. (2013).
\newblock Gaussian graphical model estimation with false discovery rate
  control.
\newblock {\em The Annals of Statistics}, 41(6):2948--2978.

\bibitem[Liu and Shao, 2014]{liu2014phase}
Liu, W. and Shao, Q.-M. (2014).
\newblock Phase transition and regularized bootstrap in large-scale $t$-tests
  with false discovery rate control.
\newblock {\em The Annals of Statistics}, 42(5):2003--2025.

\bibitem[Ma et~al., 2020]{ma2020global}
Ma, R., Cai, T.~T., and Li, H. (2020).
\newblock Global and simultaneous hypothesis testing for high-dimensional
  logistic regression models.
\newblock {\em Journal of the American Statistical Association}, pages 1--15.

\bibitem[Negahban et~al., 2009]{negahban2009unified}
Negahban, S., Yu, B., Wainwright, M.~J., and Ravikumar, P.~K. (2009).
\newblock A unified framework for high-dimensional analysis of $ m $-estimators
  with decomposable regularizers.
\newblock In {\em Advances in neural information processing systems}, pages
  1348--1356. Citeseer.

\bibitem[Nelder and Wedderburn, 1972]{nelder1972generalized}
Nelder, J.~A. and Wedderburn, R.~W. (1972).
\newblock Generalized linear models.
\newblock {\em Journal of the Royal Statistical Society: Series A},
  135(3):370--384.

\bibitem[Owen, 2005]{owen2005variance}
Owen, A.~B. (2005).
\newblock Variance of the number of false discoveries.
\newblock {\em Journal of the Royal Statistical Society: Series B},
  67(3):411--426.

\bibitem[Shao, 2003]{Shao2003}
Shao, J. (2003).
\newblock {\em Mathematical Statistics 2ed}.
\newblock Springer.

\bibitem[Tibshirani, 1996]{tibshirani1996regression}
Tibshirani, R. (1996).
\newblock Regression shrinkage and selection via the lasso.
\newblock {\em Journal of the Royal Statistical Society: Series B},
  58(1):267--288.

\bibitem[Van De~Geer et~al., 2014]{van2014asymptotically}
Van De~Geer, S., B{\"u}hlmann, P., Ritov, Y., and Dezeure, R. (2014).
\newblock On asymptotically optimal confidence regions and tests for
  high-dimensional models.
\newblock {\em The Annals of Statistics}, 42(3):1166--1202.

\bibitem[Xia et~al., 2020]{xia2020revisit}
Xia, L., Nan, B., and Li, Y. (2020).
\newblock A revisit to de-biased lasso for generalized linear models.
\newblock {\em arXiv preprint arXiv:2006.12778}.

\bibitem[Xia et~al., 2018]{xia2018joint}
Xia, Y., Cai, T.~T., and Li, H. (2018).
\newblock Joint testing and false discovery rate control in high-dimensional
  multivariate regression.
\newblock {\em Biometrika}, 105(2):249--269.

\bibitem[Ye and Zhang, 2010]{ye2010rate}
Ye, F. and Zhang, C.-H. (2010).
\newblock Rate minimaxity of the lasso and dantzig selector for the lq loss in
  lr balls.
\newblock {\em Journal of Machine Learning Research}, 11(Dec):3519--3540.

\bibitem[Yu et~al., 2021]{yuyi2018}
Yu, Y., Bradic, J., and Samworth, R. (2021).
\newblock Confidence intervals for high-dimensional cox models.
\newblock {\em Statistica Sinica}, 31(1):1038--1068.

\bibitem[Zhang and Zhang, 2014]{zhang2014confidence}
Zhang, C.-H. and Zhang, S.~S. (2014).
\newblock Confidence intervals for low dimensional parameters in high
  dimensional linear models.
\newblock {\em Journal of the Royal Statistical Society: Series B},
  76(1):217--242.

\bibitem[Zhang and Chen, 2021]{zhang2020concentration}
Zhang, H. and Chen, S.~X. (2021).
\newblock Concentration inequalities for statistical inference.
\newblock {\em Communications in Mathematical Research}, 37(1):1--85.

\bibitem[{Zhang} and {Jia}, 2022]{zhang2022elastic}
{Zhang}, H. and {Jia}, J. (2022).
\newblock Elastic-net regularized high-dimensional negative binomial
  regression: Consistency and weak signals detection.
\newblock {\em Statistica Sinica}, 32(1).

\bibitem[Zhao and Yu, 2006]{zhao2006model}
Zhao, P. and Yu, B. (2006).
\newblock On model selection consistency of lasso.
\newblock {\em Journal of Machine learning research}, 7(Nov):2541--2563.

\end{thebibliography}

\newpage
\appendix

\begin{center}
{\huge Supplementary Materials:}
\end{center}
\begin{center}
{\large Directional FDR Control for Sub-Gaussian Sparse GLMs}
\end{center}

\setcounter{table}{0}   
\renewcommand{\theequation}{\thesection.\arabic{equation}}
\renewcommand{\thelemma}{\thesection.\arabic{lemma}}
\renewcommand{\thetheorem}{\thesection.\arabic{theorem}}
\section{Some useful lemmas}
The following lemmas are the concentration inequalities for sub-Gaussian and bounded random variables, see \cite{zhang2020concentration} and references therein.
\begin{lemma}[Proposition 3.2 in \cite{zhang2020concentration}]\label{rancon}
Consider the random variables $\{Y_i\}_{i = 1}^n$ are independent distributed as $\{{\rm{subG}}(\sigma_i^2)\}_{i = 1}^n$ with $C_b^2 := \max_{1 \leq i \leq n} \sigma_{i}^2  < \infty$. Let $\bm w := ({w_1}, \cdots ,{w_n})^T \in {\mathbb{R}^n}$ be a non-random vector and define $S_n^w = :\sum_{i = 1}^n {{w_i}{Y_i}} $. Then,
\[P( {|S_n^W - {\mE}S_n^W| \ge t} ) \le  2e^{-{ t^{2}}/({2C_{ b}^{2}\|\bm w\|_{2}^{2}})}.\]
\end{lemma}

\begin{lemma}[Hoeffding's inequality, Corollary 2.1 in \cite{zhang2020concentration}]\label{hoeffd}
Let $X_1,\ldots,X_n$ be independent centered random variables with
$$\mE X_i=0,a_i\leq X_i \leq b_i,i=1,\ldots,n.$$
Then we have
\begin{itemize}
\item[{(a)}] Hoeffding's Lemma: for $u>0$, we have
$\mE \exp\left\{u\sum_{i=1}^n X_i\right\} \leq
\exp \left \{ \frac{1}{8}u^2\sum_{i=1}^n(b_i-a_i)^2 \right\}$.

\item[{(b)}] Hoeffding's inequality: for $t>0$, we have
$\mP \left(\left|\sum_{i=1}^n X_i\right|\geq t \right) \leq
2 \exp \left\{\frac{-2t^2}{\sum_{i=1}^n (b_i-a_i)^2} \right\}$.
\end{itemize}
\end{lemma}

\begin{lemma}[Corollary 7.5 in \cite{zhang2020concentration}]\label{maxi}
Let $\{X_{i}\}_{i=1}^n$ be independent r.vs on  $\mathcal{X}$ and $\{f_{i}\}_{i=1}^n$ be real-valued functions on $\mathcal{X}$ which satisfy
$\mE f_{j}(X_{i})=0,~\vert f_{j}(X_{i})\vert \le a_{ij}$ for all $j=1,...,p$ and all $i=1,...,n$. Then
\begin{center}
$\mE( \underset{1\le j\le p}{\max}| \sum\limits_{i=1}^{n}f_{j}(X_{i})|) \le [2\log(2p)]^{1 / 2}\underset{1\le j\le p}{\max}(\sum\limits_{i=1}^{n}a_{ij}^{2})^{1 / 2}.$
\end{center}
\end{lemma}
%

The following two lemmas give the properties of the function $G(t)$.

\begin{lemma}[Lemma 7.1 in \cite{javanmard2019false}]\label{Gt1}
 Let $G(t)=2(1-\Phi(t))$. For all $t\geq0$, we have
\begin{equation*}
\frac{2}{t+1/t} \phi(t)< G(t) <\frac{2}{t}\phi(t),
\end{equation*}
where $\phi(t) = e^{-t^2/2}/\sqrt{2\pi}$ is the standard Gaussian density. We further have
\begin{equation*}
\lim_{t\rightarrow \infty} \frac{G(t)}{\sqrt{\frac{2}{\pi}}\frac{e^{-t^2/2}}{t}}=1.
\end{equation*}
\end{lemma}

\begin{lemma}[Proposition 1 in \cite{ma2020global}]\label{Gt2}
 For $0\leq t \leq \sqrt{2\log p}$, uniformly we have
\begin{equation*}
\frac{G(t+o(1/\sqrt{\log p}))}{G(t)}=1+o(1).
\end{equation*}
\end{lemma}
\section{Proof of Theorem \ref{lassobeta1}}
 To proceed the proof, we require the Lemma 1 in \cite{dedieu2019an}, which resembles the high-probability first order conditions (KKT conditions) in Lasso linear models from the estimating equations.

\begin{lemma}\label{KKT-LIKE}
Let $g_{j}\sim \mathrm{subG}(\sigma^2)$ for $i=1,\cdots, p$. Denote $(g_{(1)}, \cdots, g_{(p)})$ as a non-increasing rearrangement of $(|g_{1}|, \cdots,|g_{p}|)$ and define the coefficients $\lambda_{j}=\sqrt{\log (2 p / j)}, j=1, \cdots p .$ For $\delta \in(0, \frac{1}{2}),$ it holds
$$
\mP\left(\sup _{j=1, \ldots, p}\{\frac{g_{(j)}}{\sigma \lambda_{j}^{(p)}}\} \leq 12 \sqrt{\log (1 / \delta)}\right)\ge 1-\delta.
$$
\end{lemma}
Grounded on Lemma \ref{KKT-LIKE}, we first show the following result which indicates that the difference $\Delta=\widehat\bbeta-{\bbeta^0}$ belongs to the cones condition defined in Definition \ref{new_REC}.
\begin{lemma}\label{cone}
Suppose that the data $(\bX, \bY):=\{(\bx_i, y_i)\}_{i=1}^{n}$ satisfies {(C4)(i)}. Under {(C1)} and {(C2)} for the design and the parameter, for any $\delta \in\left(0, \frac{1}{n}\right)$ and $\alpha \geq 2$, if we take $\lambda = 12 \alpha \sigma L \sqrt{\frac{\log (2p e /s_0)}{n} \log \left( \frac{1}{\delta}\right)}$, then the following cone condition holds with probability as least $1 - \delta$:
\begin{equation*}
\Delta  \in \Lambda \left({S}, \gamma_{1}^{*}:=\frac{\alpha}{\alpha-1}, \gamma_{2}^{*}:=\frac{\sqrt{s_0}}{\alpha-1}\right).
\end{equation*}
\end{lemma}
\begin{proof}
On the one hand, the conditional Kullback-Leibler divergence between ${\bbeta ^0}$ and $ \bbeta$ is defined by
\[{K_n}({\bbeta ^0},\bbeta; \bX):= {\mE}[\ell_n ({\bbeta ^0}) - \ell_n (\bbeta )|\bX]\ge 0,  \bbeta  \in \mathbb{R}^p.\]
The un-conditional Kullback-Leibler divergence is given by ${\mE}[{K_n}({\bbeta^0},\bbeta; \bX)]={\mE}[\ell_n ({\bbeta ^0}) - \ell_n (\bbeta )]$.

For GLMs with canonical link, we have
\begin{align}\label{eq-kl}
{K_n}({\bbeta ^0},\widehat\bbeta{;\bX}) =\frac{1}{n}\sum_{i=1}^{n}\left\{({\mE}{y_i})\bx_i^{T}({\bbeta ^0}-\widehat\bbeta)+b( {\bx}_i^{T}\widehat\bbeta)-b\left( {\bx}_i^{T}{\bbeta ^0}\right)\right\}.
\end{align}
Substitute ${\mE}{y_i} = {y_i} + ({\mE}{y_i} - y_i)$ into (\ref{eq-kl}), we get
\begin{align*}
{K_n}({\bbeta ^0},\widehat\bbeta{;\bX})&= \frac{1}{n}\sum_{i=1}^{n}\left\{{y_i}{\bx}_i^{T}({\bbeta ^0}-\widehat\bbeta)+b( {\bx}_i^{T}{\widehat\bbeta})-b\left( {\bx}_i^{T}{\bm\beta^0}\right)\right\}+\frac{1}{n}\sum_{i=1}^{n}\left\{({\mE}{y_i} - y_i){\bx}_i^{T}({\bbeta ^0}-\widehat\bbeta)\right\} .\\
& = \ell_n (\widehat \bbeta ) - \ell_n ({\bbeta ^0} ) +\frac{1}{n}\sum_{i=1}^{n}({\mE}{y_i} - y_i){\bx}_i^{T}({\bbeta ^0}-\widehat\bbeta).
\end{align*}

On the other hand, the definition of $\widehat{\bbeta}$ implies
$$\ell_n(\widehat \bbeta) +\lambda\|\widehat\bbeta\|_{1} \le \ell_n({\bbeta ^0}) + \lambda\|{\bbeta ^0}\|_{1} \Rightarrow \ell_n(\widehat \bbeta)-\ell_n({\bbeta ^0})  \le \lambda\|{\bbeta ^0}\|_{1}-\lambda\|\widehat\bbeta\|_{1},$$
which gives
\begin{align}\label{eq:KLinq}
{K_n}({\bbeta ^0},\widehat\bbeta{;\bX})& \le \lambda\|{\bbeta ^0}\|_{1}-\lambda\|\widehat\bbeta\|_{1} +\frac{1}{n}\sum_{i=1}^{n}(y_i-{\mE}{y_i}){\bx}_i^{T}\Delta\nonumber\\
& =\lambda\|{\bbeta_S^0}\|_{1}-\lambda\|\widehat\bbeta_S\|_{1}-\lambda\|(\widehat\bbeta_{S^c}-{\bbeta_{S^c}^0})+{\bbeta_{S^c}^0}\|_{1} +\frac{1}{n}\sum_{i=1}^{n}(y_i-{\mE}{y_i}){\bx}_i^{T}\Delta\nonumber\\
&\le \lambda\|{\Delta_S}\|_{1}-\lambda\|\Delta_{S^c}\|_{1} +\frac{1}{n}\sum_{i=1}^{n}(y_i-{\mE}{y_i}){\bx}_i^{T}\Delta.
\end{align}
It remains to control the upper bound of $\frac{1}{n}\sum_{i=1}^{n}(y_i-{\mE}{y_i}){\bx}_i^{T}\Delta$ via Lemma \ref{KKT-LIKE}. To see this, denote $\bm g :=\frac{1}{\sqrt{n}}\sum_{i=1}^{n}(y_i-{\mE}{y_i}){\bx}_i^{T}$ and then {(C4)} gives $g_{j}\sim \mathrm{subG}(\sigma^2L^2)$. Let $\left(g_{(1)}, \ldots, g_{(p)}\right)$ be a non-increasing rearrangement of $\left(\left|g_{1}\right|, \ldots,\left|g_{p}\right|\right)$. We have,
\begin{align}\label{Group_Las_thm_1_2}
\frac{1}{\sqrt{n}}\sum_{i=1}^{n}(y_i-{\mE}{y_i}){\bx}_i^{T}\Delta &= \sum_{j = 1}^p g_j \Delta_j \leq \sum_{j = 1}^p |g_j| |\Delta_j|  = \sum_{j=1}^{p} \frac{g_{(j)}}{ \sigma \lambda_{j}} \sigma L \lambda_{j}\left|\Delta_{(j)}\right| \nonumber\\
& \leq  \sigma L \sup _{j=1, \ldots, p}\left\{\frac{g_{(j)}}{\sigma L\lambda_{j}}\right\} \sum_{j=1}^{p} \lambda_{j}\left|\Delta_{(j)}\right| \nonumber\\
[{\text{Lemma}~\ref{KKT-LIKE}}]~&{\leq} 12  \sigma L \sqrt{\log (1 / \delta)} \sum_{j=1}^{p} \lambda_{j}\left|\Delta_{(j)}\right| \nonumber\\
[\lambda_{1} \geq \ldots \geq \lambda_{p} \text { and }\left|\Delta_{1}\right| \geq \ldots \geq\left|\Delta_{p}\right|]  & {\leq} 12 \sigma L \sqrt{\log (1 / \delta)} \sum_{j=1}^{p} \lambda_{j}\left|\Delta_{j}\right| \nonumber\\
& \leq 12\sigma L \sqrt{\log (1 / \delta)}(\sum_{j=1}^{s_0} \lambda_{j}\left|\Delta_{j}\right|+\lambda_{s_0}\left\|\Delta_{{S}^c}\right\|_{1})
\end{align}
holds with the probability at least $1 - \delta$. Recall that $\lambda_{j}=\sqrt{\log (2 p / j)}$. Then Stirling's approximation $n! \ge {(n/e)^n}$ shows
\begin{equation*}
\sum_{j=1}^{s_0} \lambda_{j}^{2}=\sum_{j=1}^{s_0} \log (2 s_0 / j)=s_0 \log (2 p)-\log \left(s_0 !\right) \leq s_0 \log (2 p)-s_0 \log \left(s_0/ e\right)=s_0 \log \left(2 p e / s_0\right).
\end{equation*}
For \eqref{Group_Las_thm_1_2}, Cauchy's inequality gives
\begin{equation}\label{Cauchy}
\sum_{j=1}^{s_0} \lambda_{j}\left|\Delta_{j}\right|
\leq \left\|\Delta_{{S}}\right\|_{2}\sqrt{\sum_{j=1}^{s_0} \lambda_{j}^{2}} \leq \left\|\Delta_{{S}}\right\|_{2}\sqrt{s_0 \log \left(2 p e /s_0\right)}.
\end{equation}
Put $\lambda := 12 \alpha \sigma L \sqrt{\frac{\log (2p e /s_0)}{n} \log \left( \frac{1}{\delta}\right)}$. With $\lambda_{s_0} \leq \sqrt{\log \left(2 p e / s_0\right)}$, \eqref{Group_Las_thm_1_2} and \eqref{Cauchy}, the equation \eqref{eq:KLinq} turns into
\begin{align}\label{Group_Las_thm_1_3}
0\le {K_n}({\bbeta ^0},\widehat\bbeta{;\bX} ) & \le \lambda\|{\Delta_S}\|_{1}-\lambda\|\Delta_{S^c}\|_{1} +\frac{1}{n}\sum_{i=1}^{n}(y_i-{\mE}{y_i}){\bx}_i^{T}\Delta\nonumber\\
& \leq \lambda\|{\Delta_S}\|_{1}-\lambda\|\Delta_{S^c}\|_{1} +12 \alpha \sigma L \sqrt{\frac{\log (2p e /s_0)}{n} \log \left( \frac{1}{\delta}\right)} [\sqrt{s_0} \| \Delta_{{S}} \|_2 + \| \Delta_{{S}^c} \|_1 ]\nonumber\\
& = \frac{\lambda}{\alpha} \left( \sqrt{s_0} \| \Delta_{{S}} \|_2 + \|\Delta_{{S}^c} \|_1\right) + \lambda \| \Delta_{{S}} \|_1 - \lambda \| \Delta_{{S}^c} \|_1.
\end{align}
with probability at least $1 - \delta$. Then, \eqref{Group_Las_thm_1_3} gives
\begin{equation}\label{coneproof}
\left\|\Delta_{{S}^c}\right\|_{1} \leq \frac{\alpha}{\alpha-1}\left\|\Delta_{{S}}\right\|_{1}+\frac{\sqrt{s_0}}{\alpha-1}\left\|\Delta_{{S}}\right\|_{2}
\end{equation}
with probability at least $1 - \delta$. The claim of Lemma \ref{cone} follows.
\end{proof}
\begin{proof}[proof of Theorem \ref{lassobeta1}]
Having get the cone condition \eqref{coneproof}, we can derive oracle inequality like the procedure in \cite{bickel2009simultaneous}. From \eqref{Group_Las_thm_1_3}, observe that
\begin{align}\label{eq:klupper}
{K_n}({\bbeta ^0},\widehat\bbeta{;\bX}) & \leq \frac{\lambda}{\alpha} \left( \sqrt{s_0} \| \Delta_{{S}} \|_2 + \|\Delta_{{S}^c} \|_1\right) + \lambda \| \Delta_{{S}} \|_1 - \lambda \| \Delta_{{S}^c} \|_1 \nonumber\\
[{\alpha} > 1]    & < \frac{\lambda}{\alpha} \sqrt{s_0} \|\Delta_{{S}}\|_2 + \lambda \| \Delta_{{S}} \|_1 \leq 2 \lambda \sqrt{s_0} \| \Delta_{{S}} \|_2 \leq 2 \lambda \sqrt{s_0} \| \Delta \|_2.
\end{align}
For the last inequality, we use the Cauchy's inequality.

Now, apply Taylor expansion to \eqref{eq-kl} on $b\left(t\right)$ at $t= {\bx}_i^{T}{\bbeta ^0}$, we have
\begin{align}\label{eq:kl2}
&     {K_n}({\bbeta ^0},\widehat\bbeta{;\bX}) =\frac{1}{n}\sum_{i=1}^{n}\left\{b( {\bx}_i^{T}{\bbeta ^0}){\bx}_i^{T}({\bbeta ^0}-\widehat\bbeta)+b( {\bx}_i^{T}\widehat\bbeta)-b( {\bx}_i^{T}{\bbeta ^0})\right\}\nonumber\\
&=\frac{1}{n}\sum_{i=1}^{n}\left\{\dot b( {\bx}_i^{T}{\bbeta ^0}){\bx}_i^{T}({\bbeta ^0}-\widehat\bbeta)-\dot b( {\bx}_i^{T}{\bbeta ^0}){\bx}_i^{T}({\bbeta ^0}-\widehat\bbeta)+\ddot b(s{\bx}_i^{T}{\bbeta ^0}+(1-s){\bx}_i^{T}\widehat\bbeta)\right\}\nonumber\\
&=\frac{1}{n}\sum_{i=1}^{n}\ddot b(s_i{\bx}_i^{T}{\bbeta ^0}+(1-s_i){\bx}_i^{T}\widehat\bbeta){[ {{\bx}_i^{T}({\bbeta ^0}-\widehat\bbeta)}]^2}.
\end{align}

Recall the first inequality in \eqref{eq:KLinq} and we further have
\begin{align*}
{K_n}({\bbeta ^0},\widehat\bbeta{;\bX} )& \le \lambda\|{\bbeta ^0}\|_{1}-\lambda\|\widehat\bbeta\|_{1} +\frac{1}{n}\sum_{i=1}^{n}(y_i-{\mE}{y_i}){\bx}_i^{T}\Delta \nonumber\\
&\le \lambda\|{\bbeta ^0}\|_{1}-\lambda\|\widehat\bbeta\|_{1} + \frac{12 \sigma L}{\sqrt n} \sqrt{\log (1 / \delta)} \sum_{j=1}^{p} \lambda_{j}\left|\Delta_{j}\right|,
\end{align*}
which shows that
\begin{equation*}
\|\widehat\bbeta\|_{1} \leq \|{\bbeta ^0}\|_{1}+{\sum\limits_{j = 1}^p {\sqrt {\frac{{\log (2p/j)}}{{\alpha^2\log (2pe/{s_0})}}} } \left| {{\Delta _j}} \right|}\le \|{\bbeta ^0}\|_{1}+\mathop {\max }\limits_j \sqrt {\frac{{\log (2p/j)}}{{\alpha^2\log (2pe/{s_0})}}}(\|{\bbeta ^0}\|_{1}+\|\widehat\bbeta\|_{1}).
\end{equation*}
By then we have verified
$\|\widehat\bbeta\|_{1} \leq\frac{{1 + \mathop {\max }\limits_j \sqrt {\frac{{\log (2p/j)}}{{{\alpha ^2}\log (2pe/{s_0})}}} }}{{1 - \mathop {\max }\limits_j \sqrt {\frac{{\log (2p/j)}}{{{\alpha ^2}\log (2pe/{s_0})}}} }} = \frac{{1 + \sqrt {\frac{{\log (2p)}}{{{\alpha ^2}\log (2pe/{s_0})}}} }}{{1 - \sqrt {\frac{{\log (2p)}}{{{\alpha ^2}\log (2pe/{s_0})}}} }}=:K_\alpha$. Thus, we get
$$s{\bx}_i^{T}{\bbeta ^0}+(1-s){\bx}_i^{T}\widehat\bbeta \le L(B+K_\alpha).$$
by {(C1)} and {(C2)}. Therefore, \eqref{eq:kl2} leads to
$$ {K_n}({\bbeta ^0},\widehat\bbeta{;\bX}) \ge \mathop {\inf }\limits_{\left| t \right| \le L(B + K_\alpha)} \ddot b(t)\frac{1}{n}\sum_{i=1}^{n}{[ {{\bx}_i^{T}({\bbeta ^0}-\bbeta)}]^2}|_{{\bbeta=\widehat\bbeta}}.$$
Then, \eqref{eq:klupper} and $\mathrm{RE}(s_0, \gamma, \mE(\bx \bx^{T}))$ with the restricted minimal eigenvalue $\kappa^* >0$ imply
\begin{align*}
2 \lambda \sqrt{s_0} \| {\bbeta ^0}-\widehat\bbeta \|_2  \ge \mE [{K_n}({\bbeta ^0},\bbeta{;\bX})]|_{{\bbeta=\widehat\bbeta}}&=\mathop {\inf }\limits_{\left| t \right| \le L(B + K_\alpha)} \ddot b(t){\mE [ {{\bx}^{T}({\bbeta ^0}-\bbeta)}]^2}|_{{\bbeta=\widehat\bbeta}}\\
& \ge  \kappa^* \mathop {\inf }\limits_{\left| t \right| \le L(B + K_\alpha)}\ddot b(t)\| {\bbeta ^0}-\widehat\bbeta \|_2^2
\end{align*}
which proves that $\| {\bbeta ^0}-\widehat\bbeta \|_2 \le \frac{2 \sqrt{s_0}}{\kappa^* \mathop {\inf }\limits_{\left| t \right| \le L(B + K_\alpha)}\ddot b(t)}\lambda $. Cauchy's inequality also gives
\begin{equation*}
\| {\bbeta ^0}-\widehat\bbeta\|_1 \le \sqrt{s_0}\| {\bbeta ^0}-\widehat\bbeta\|_2 \le \frac{2{s_0}}{\kappa^* \mathop {\inf }\limits_{\left| t \right| \le L(B + K_{\alpha})}\ddot b(t)}\lambda.
\end{equation*}
\end{proof}

\section{Proof of Proposition \ref{main-thm1}}
We first try to break down the debiased estimator defined in $\eqref{debia}$ into several parts. For $j\in[p]$, let $\dot{\ell}_j(\widehat{\bbeta})$ be the $j$-th component of the score function $\dot{\ell}(\bbeta)$ in \eqref{score}.
Let $\ddot{\ell}_j(\widehat{\bbeta})\in \mathbb{R}^{p}$ be the vector of whose $k$-th component is $\frac{\partial^{2} \ell(\bbeta)}{\partial \bbeta_{k} \partial \bbeta_{j}}$. By Taylor expansion, for each $j\in[p]$, there exists vector $\widetilde{\bbeta}_j$ on the segment between $\widehat{\bbeta}$ and $\bbeta^0$ such that
\begin{equation}\label{Taylor}
\dot \ell _j(\widehat{ \bbeta}) = {\dot \ell _j}({\bbeta ^0})
+ {\ddot \ell _{_j}}{({\widetilde{\bbeta} _j})^T}(\widehat{\bbeta} - {\bbeta ^0}).
\end{equation}
Construct matrix $\boldsymbol{W}(\widetilde{\bbeta}) \in \mathbb{R}^{p \times p}$,
whose $j$-th row is $\ddot{\ell}_{j}(\widetilde{\bbeta}_{j})^T$. Then, for any $\boldsymbol{c} \in \mathbb{R}^{p}$, the decomposition of the debiased estimator is
$\|\boldsymbol{c}\|_{1}=1$,
\begin{equation}\label{decom}
\begin{split}
\boldsymbol{c}^T\left(\widehat{\bbeta}^d-\bbeta^{0}\right)
=&\boldsymbol{c}^T\left\{\widehat{\bbeta}+\widehat{\bTheta}
 \dot{\ell}(\widehat{\bbeta})-\bbeta^{0}\right\} \\
=&\boldsymbol{c}^T \bSigma^{-1} \dot{\ell}\left(\bbeta^{0}\right)+
\boldsymbol{c}^T\left(\widehat{\bTheta}-\bSigma^{-1}\right)
\dot{\ell}\left(\bbeta^{0}\right)+\boldsymbol{c}^T
 \widehat{\bTheta}\left\{\dot{\ell}(\widehat{\bbeta})-
 \dot{\ell}\left(\bbeta^{0}\right)\right\}+
 \boldsymbol{c}^T\left(\widehat{\bbeta}-\bbeta^{0}\right) \\
=&\boldsymbol{c}^T \bSigma^{-1} \dot{\ell}\left(\bbeta^{0}\right)+
\boldsymbol{c}^T\left(\widehat{\bTheta}-
\bSigma^{-1}\right) \dot{\ell}\left(\bbeta^{0}\right)+
\boldsymbol{c}^T\left\{\widehat{\bTheta} \boldsymbol{W}(\widetilde{\bbeta})
+\boldsymbol{I}_p\right\}\left(\widehat{\bbeta}-\bbeta^{0}\right)\\
=&\Delta_1+\Delta_2+\Delta_3.
\end{split}
\end{equation}

In the following subsections, we will prove that after multiplying \eqref{decom} by $\sqrt{n}$, $\Delta_1$ is the main part of it, whose asymptotic distribution is a normal distribution, and $\Delta_2, \Delta_3$ are $o_p\left(\frac{1}{\sqrt{n}}\right)$ which can be ignored in the asymptotic sense.

Before proving Proposition \ref{main-thm1}, we first give some useful properties of the score function $\dot{\ell}(\bbeta^0)$ through the following lemma.
\begin{lemma}\label{inilasso}
Suppose {(C1)} and {(C4)(i)} hold, then for any $t>0$
\begin{equation*}
\mP \{\|\dot{\ell}(\bbeta^0)\|_{\infty} > t\} \leq 2p e^{-nt^2/(2C^4)}.
\end{equation*}
Furthermore, we set $t=O\left(\sqrt{\frac{\log p}{n}}\right)$, then we have
 $\|\dot{\ell}(\bbeta^0) \|_{\infty}=O_p\left(\sqrt{\frac{\log p}{n}}\right).$
\end{lemma}

\begin{proof}
Utilize \eqref{score}, Lemma \ref{rancon} and Condition {(C1)} and {(C4)(i)}, we have
\begin{equation*}
\mP \{\|\dot{\ell}(\bbeta^0)\|_{\infty} > t\}
\leq  p \times \mP \left( \left|\sum_{i=1}^{n}x_{ij}\left(y_i-\dot{b}_j
(\bx_i^T\bbeta)\right)\right| > nt \right)
 \leq  2p\times e^{-nt^2/(2C^4)},
\end{equation*}
where $C$ is a constant, $j\in[p]$.
\end{proof}

\subsection{The main part $\Delta_1$}
In \eqref{decom}, let $\Delta_1$ time $\sqrt{n}$ and we get
\begin{equation}\label{maindecom}
\sqrt{n}\bc^T{\bSigma}^{-1}\dot{\ell}\left(\bbeta^{0}\right)
= \frac{1}{\sqrt{n}}\sum_{m=1}^{n}\bc^T{\bSigma}^{-1}\bx_m[y_m-\dot{b}
(\bx_m^T\bbeta^0))].
\end{equation}
The following proposition gives the asymptotic normality of the above expression.
\begin{proposition}\label{prop1}
Suppose {(C1), (C2)},{(C4)(i)}, (C5)(i) and {(C6)} hold. Let $c\in \mathbb{R}^p$ satisfying
$\|c\|_1 = 1$ and $\bc^T{\bSigma}^{-1}\bc
\rightarrow \sigma^2 \in (0,\infty)$. Then, as $n,p\rightarrow \infty$,
$$
\sqrt{n}\bc^T{\bSigma}^{-1}\dot{\ell}(\bbeta^0)
\stackrel{d}{\to} \mathcal{N}(0, \sigma^2).
$$
\end{proposition}

\begin{proof}
Using Lindeberg-Feller central limit theorem \citep{durrett2019}, for $1\leq m \leq n$, let
$$x_{n,m}=\frac{1}{\sqrt{n}}\bc^T{\bSigma}^{-1}\bx_m\left(y_m-\dot{b}
(\bx_m^T\bbeta^0)\right).$$
So $x_{n,m}(1\leq m \leq n)$ are mutually independent random vectors. For $1\leq m \leq n$, we have
\begin{equation}\label{I1-1}
\mE x_{n,m}=\mE \left\{ \mE \left[x_{n,m} |\bx_m\right]\right\}
=\mE \left\{\frac{1}{\sqrt{n}}\bc^T{\bSigma}^{-1}\bx_m\mE
\left[\left(y_m-\dot{b}
(\bx_m^T\bbeta^0)\right)\bigg|\bx_m\right]\right\}=0.
\end{equation}

Then we have
\begin{equation}\label{I1-2}
\begin{split}
\sum_{m=1}^{n}\mE x_{n,m}^2 &=\sum_{m=1}^{n}\frac{1}{n} \mE \left\{\mE \left[\left(
\bc^T{\bSigma}^{-1}\bx_m\right)^2\left(y_m-\dot{b}
(\bx_m^T\bbeta^0)\right)^2\bigg|\bx_m\right] \right\}\\
&=\sum_{m=1}^{n}\frac{1}{n} \mE \left\{\left(\bc^T{\bSigma}^{-1}\bx_m
\right)^2  \mE \left[\left(y_m-\dot{b}
(\bx_m^T\bbeta^0)\right)^2\bigg|\bx_m\right] \right\}\\
&=\sum_{m=1}^{n}\frac{1}{n} \mE \left\{\left(\bc^T{\bSigma}^{-1}\bx_m
\right)^2 \ddot{b}(\bx_m^T\bbeta^0) \right\}\\
&=\bc^T{\bSigma}^{-1}\left\{ \frac{1}{n}\sum_{m=1}^{n}
\mE \bx_m \bx_m^T\ddot{b}(\bx_m^T\bbeta^0) \right\}{\bSigma}^{-1}\bc=\bc^T{\bSigma}^{-1}{\bSigma}{\bSigma}^{-1}\bc
=\bc^T{\bSigma}^{-1}\bc,
 \end{split}
\end{equation}
where the penultimate equation is due to $\bSigma:=\mE\bx_i\bx_i^T\ddot{b}(\bx_i^T\bbeta^0)$.

Furthermore, for any $\varepsilon>0,0<t<1$,
\begin{equation}\label{I1-3}
\begin{split}
\sum_{m=1}^{n}\mE x_{n,m}^2 \mathbbm{1}_{\{|x_{n,m}| \geq \varepsilon\}}
&\leq \sum_{m=1}^{n}\mE x_{n,m}^2 \left|\frac{x_{n,m}}{\varepsilon} \right|^t \\
&=\sum_{m=1}^{n} \frac{1}{n} \mE \left\{\left|\bc^T{\bSigma}^{-1}\bx_m\right|^{2+t} \left|y_m-\dot{b}(\bx_m^T\bbeta^0) \right|^{2+t} \right \}\cdot
\frac{1}{(\sqrt{n}\varepsilon)^t} \\
&=\sum_{m=1}^{n} \frac{1}{n}\mE \left\{ \left|\bc^T{\bSigma}^{-1}\bx_m\right|^{2+t} \mE \left[\left|y_m-\dot{b}(\bx_m^T\bbeta^0) \right|^{2+t} \bigg |\bx_m
 \right]\right \} \cdot \frac{1}{(\sqrt{n}\varepsilon)^t} .
\end{split}
\end{equation}

For \eqref{I1-3}, we have
$|\bc^T{\bSigma}^{-1}\bx_m|\leq \|\bc\|_1\cdot
\|{\bSigma}^{-1}\|_{\mathrm{op}, 1} \cdot\|\bx_m\|_{\infty} \leq L \|{\bSigma}^{-1}\|_{\mathrm{op}, 1}$ by {(C1)}. Utilize {(C1)}, {(C2)}, {(C4)(i)} and the property of sub-Gaussian random variable, and we know
that $\mE [|y_m-\dot{b}(\bx_m^T\bbeta^0)|^{2+t} |\bx_m]$ are uniformly bounded with a certain upper bound $C$. So substitute these two pieces into \eqref{I1-3} and thus when $n, p\rightarrow \infty$,
\begin{equation}\label{I1-4}
\sum_{m=1}^{n}\mE x_{n,m}^2
\mathbbm{1}_{\{|x_{n,m}|\geq \varepsilon\}} \leq C L \|{\bSigma}^{-1}\|_{\mathrm{op}, 1}^{2 + t} \frac{1}{(\sqrt{n}\varepsilon)^t} \rightarrow 0
\end{equation}
by noticing that (C5)(i) and (C6) implies $\|{\bSigma}^{-1}\|_{\mathrm{op}, 1}^{2 + t} \cdot \frac{1}{n^{t / 2}} = o(1)$ when $t \geq 1 / 2$. Finally, from \eqref{I1-1}, \eqref{I1-2}, \eqref{I1-4} and Lindeberg-Feller central limit theorem, we know that
\begin{equation*}
\sqrt{n}\bc^T{\bSigma}^{-1}\dot{\ell}(\bbeta^0)
\stackrel{d}{\to} \mathcal{N}(0, \sigma^2).	
\end{equation*}
\end{proof}

\subsection{The remainder $\Delta_2$ and $\Delta_3$}
We first give some auxiliary lemmas. The following conclusion comes from Lemma 7.1 \citep{cai2016} and Lemma 9 \citep{yuyi2018}.
\begin{lemma}
\label{prop2-1}
Let ${\bTheta} = ({\bTheta}_1,\ldots,{\bTheta}_p)^T
= (\Theta_{ij}) \in \mathbb{R}^{p \times p}$. Let
$\widehat{{\bTheta}} = (\widehat{{\bTheta}}_1,\ldots,
\widehat{{\bTheta}}_p)^T$ be the estimator of ${\bTheta}$. Then on event
\begin{equation*}
\bigl\{\|\widehat{{\bTheta}}_j\|_1 \leq \|{\bTheta}_j\|_1, \, j\in[p]\bigr\},	
\end{equation*}
we have
\begin{equation*}
\|\widehat{{\bTheta}} - {\bTheta}\|_{\mathrm{op},\infty} \leq 12\|\widehat{{\bTheta}} - {\bTheta}\|_{\infty} \max_{j\in[p]} \sum_{i=1}^p \mathbbm{1}_{\{\Theta_{ij} \neq 0\}}.	
\end{equation*}
\end{lemma}

The following lemma is similar to Lemma 8 \citep{yuyi2018}, which is used to bound the probability of event $\mathcal{A}$ in \eqref{eventA}.

\begin{lemma}\label{prop2-2}
Suppose {(C1)--(C4)} hold. Set
$\lambda\asymp\sqrt{\frac{\log (p/s_0)}{n}}$
in $\eqref{lassoes}$ and we get estimator $\widehat{\bbeta}$. Then we have
\begin{equation*}
\|{\bSigma}^{-1}{\bSigma}_n -
{I}_p\|_{\infty} =O_p\left\{ \left\|{\bSigma}^{-1}\right\|_{\mathrm{op}, 1} \cdot s_0\sqrt{\frac{\log (p/s_0)}{n}} \right\}.	
\end{equation*}
\end{lemma}

\begin{proof}
First we note that
\begin{equation}\label{Sigma-1}
\|{\bSigma}^{-1}{\bSigma}_n -
{I}_p\|_{\infty} \leq \left\|{\bSigma}^{-1}\right\|_{\mathrm{op}, 1}
\cdot \| {\bSigma}_n- {\bSigma} \|_{\infty}.
\end{equation}
From $\bSigma:=\mE\bx_i\bx_i^T\ddot{b}(\bx_i^T\bbeta^0)$ and $\bSigma_n:=
\frac{1}{n} \sum_{i=1}^{n}\bx_i\bx_i^T\ddot{b}(\bx_i^T\widehat{\bbeta})$, we know that
\begin{equation}\label{Sigma-}
\| {\bSigma}_n- {\bSigma} \|_{\infty} \leq
\|\ddot{\ell}(\widehat{\bbeta})-\ddot{\ell}(\bbeta^0) \|_{\infty}
+\|\ddot{\ell}(\bbeta^0)-
\mE \{\ddot{\ell}(\bbeta^0)\} \|_{\infty}.
\end{equation}
For the first part of the right hand side of \eqref{Sigma-}, its $jk$-th $(j, k \in [p])$ element is
\begin{equation*}
\left|\frac{1}{n}\sum_{i=1}^n x_{ij}x_{ik}\left[\ddot{b}(\bx_i^T\bbeta^0)
-\ddot{b}(\bx_i^T\widehat{\bbeta}) \right]\right|
\leq \frac{1}{n}\sum_{i=1}^n \left|x_{ij}x_{ik}\right| \cdot C_R
\left| \bx_i^T\bbeta^0-\bx_i^T\widehat{\bbeta}\right| \leq \frac{1}{n}\sum_{i=1}^n C^3\|\bbeta^0- \widehat{\bbeta}\|_1,
\end{equation*}
where the first inequality is due to the inequality $ \bx_i^T\bbeta^0-\bx_i^T\widehat{\bbeta}\le L\|\bbeta^0- \widehat{\bbeta}\|_1 \le L(B+K_\alpha)=:R$ in the proof of and the Lipschitz property of the condition {(C4)(ii)} and the second inequality is due to the condition {(C1)}.

So from Theorem \ref{lassobeta1} we know that
\begin{equation}\label{ell-2}
\|\ddot{\ell}(\widehat{\bbeta})-\ddot{\ell}(\bbeta^0) \|_{\infty}=
O_p\left(s_0\sqrt{\frac{\log (p/s_0)}{n}}\right).
\end{equation}
For the second part of \eqref{Sigma-}, its $jk$-th $(j,k\in[p])$ element is
$$\frac{1}{n}\sum_{i=1}^n \left\{x_{ij}x_{ik}\ddot{b}(\bx_i^T\bbeta^0)-
\mE \left[x_{ij}x_{ik}\ddot{b}(\bx_i^T\bbeta^0)\right] \right\}.$$

From the conditions {(C1),(C2)} and {(C4)(i)}, there is a constant $C_1$, so that for any $j,k\in[p]$, we have $\left |x_{ij}x_{ik}\ddot{b}(\bx_i^T\bbeta^0)\right|<C_1$.
Thus,
\begin{equation*}
\left |x_{ij}x_{ik}\ddot{b}(\bx_i^T\bbeta^0)-
 \mE \left[x_{ij}x_{ik}\ddot{b}(\bx_i^T\bbeta^0)\right] \right|<2C_1.	
\end{equation*}
From the Hoeffding inequality in Lemma \ref{hoeffd}, for $t>0$, we have
\begin{equation*}
\mP \left(\frac{1}{n}\left|\sum_{i=1}^n\left\{x_{ij}x_{ik}\ddot{b}(\bx_i^T\bbeta^0)-
\mE \left[x_{ij}x_{ik}\ddot{b}(\bx_i^T\bbeta^0)\right]  \right\} \right|  \geq t \right) 2\exp\left\{ \frac{-2n^2t^2}{16nC_1^2} \right\} =2\exp\left\{ \frac{-nt^2}{8C_1^2} \right\}.
\end{equation*}
So we have
\begin{equation}\label{ell-4}
\mP \left( \|\ddot{\ell}(\bbeta^0)-
\mE \{\ddot{\ell}(\bbeta^0)\} \|_{\infty}\geq t  \right)
\leq p(p+1)\exp\left\{ \frac{-nt^2}{8C_1^2} \right\}.
\end{equation}
For \eqref{ell-4}, we set $t=O\left( \sqrt{\frac{\log p}{n}}\right)$ and get
\begin{equation}\label{ell-5}
\|\ddot{\ell}(\bbeta^0)-
\mE \{\ddot{\ell}(\bbeta^0)\} \|_{\infty}=
O_p\left(\sqrt{\frac{\log p}{n}}\right).
\end{equation}
Combining \eqref{Sigma-1}, \eqref{Sigma-}, \eqref{ell-2} and \eqref{ell-5} and noticing that $\sqrt{\frac{\log p}{n}} = O\left(s_0\sqrt{\frac{\log (p/s_0)}{n}}\right)$, we know that Lemma \ref{prop2-2} holds.
\end{proof}

Below we use Proposition \ref{prop2} and Proposition \ref{prop3} to bound the remaining items $\Delta_2$ and $\Delta_3$.
\begin{proposition}\label{prop2}
Suppose {(C1)--(C6)} hold. Let $\lambda\asymp\sqrt{\frac{\log (p/s_0)}{n}}$ in $\eqref{lassoes}$,
then we have Lasso estimator $\widehat{\bbeta}$. Let $$\lambda_n\asymp\left\|{\bSigma}^{-1}\right\|_{\mathrm{op}, 1}s_0\sqrt{\frac{\log p}{n}}$$ in $\eqref{Theta}$, then we have $\widehat{{\bTheta}}$, the estimation of ${\bSigma}^{-1}$. Then, we get
$$\left\|\bigl(\widehat{{\bTheta}} - {\bSigma}^{-1}\bigr)
\dot{\ell}(\bbeta^0)\right\|_{\infty} = o_p\left( \frac{1}{\sqrt{n \log p}}\right).$$
\end{proposition}
\begin{proof}
First define event
\begin{equation}\label{eventA}
 \mathcal{A} :=\{\|{\bSigma}^{-1}{\bSigma}_n -
{I}_p\|_{\infty} \leq \lambda_n\} = \{\|{\bSigma}_n{\bSigma}^{-1}
- {I}_p\|_{\infty} \leq \lambda_n\},
\end{equation}
where
\begin{equation*}
{\bSigma}_n:=-\ddot{\ell}(\widehat{\bbeta})=
\frac{1}{n}\sum_{i=1}^{n}\bx_i\bx_i^T\ddot{b}(\bx_i^T\widehat{\bbeta}).
\end{equation*}
From Lemma \ref{prop2-2}, we know the condition $\lambda_n\asymp\left\|{\bSigma}^{-1}\right\|_{\mathrm{op}, 1} s_0 \sqrt{\frac{\log p}{n}}$ can make the event $\mathcal{A}$ hold with a high probability. From the construction of
$\widehat{{\bTheta}}=(\widehat{{\bTheta}}_{1},\ldots, \widehat{{\bTheta}}_{p})^{T}$
in \eqref{Theta}, we know that, on event $\mathcal{A}$, we have
$$\|\widehat{{\bTheta}}_{j}\|_1\leq \|{\bSigma}^{-1}_{j}\|_1,j\in[p].$$
Then
$\left\|\widehat{{\bTheta}}\right\|_{\mathrm{op}, \infty}
 \leq\left\|{\bSigma}^{-1}\right\|_{\mathrm{op}, \infty}=
 \left\|{\bSigma}^{-1}\right\|_{\mathrm{op},1}.$
Also, by \eqref{Theta} we know that
$\|\widehat{{\bTheta}}{\bSigma}_n - {I}_p\|_{\infty} \leq \lambda_n.$

With the above derivation, on event $\mathcal{A}$, we have decomposition
\begin{equation}\label{prop2.1}
\begin{split}
\|\widehat{{\bTheta}} - {\bSigma}^{-1} \|_{\infty}
&=\left\|\widehat{{\bTheta}}\left( {I}_p-{\bSigma}_n
{\bSigma}^{-1}\right) +
\left( \widehat{{\bTheta}}{\bSigma}_n-{I}_p\right)
{\bSigma}^{-1} \right\|_{\infty} \\
&\leq \left\|\widehat{{\bTheta}}\left( {I}_p-{\bSigma}_n
{\bSigma}^{-1}\right)\right\|_{\infty}+
\left\|\left( \widehat{{\bTheta}}{\bSigma}_n-{I}_p\right)
{\bSigma}^{-1} \right\|_{\infty} \\
&\leq \|\widehat{{\bTheta}} \|_{\mathrm{op}, \infty}\cdot
\|{\bSigma}_n {\bSigma}^{-1}- {I}_p\|_{\infty}+
\|\widehat{{\bTheta}} {\bSigma}_n-{I}_p\|_{\infty}
\cdot \| {\bSigma}^{-1}\|_{\mathrm{op}, 1} \\
&\leq 2 \lambda_n\left\|{\bSigma}^{-1}\right\|_{\mathrm{op},1}.
\end{split}
\end{equation}
So by \eqref{prop2.1} and Lemma \ref{prop2-1}, on the event $\mathcal{A}$, we have
\begin{equation}\label{prop2.2}
\begin{split}
\left\|\bigl(\widehat{{\bTheta}} - {\bSigma}^{-1}\bigr) \dot{\ell}(\bbeta^0)\right\|_{\infty} &\leq \left\|\widehat{{\bTheta}} - {\bSigma}^{-1}\right\|_{\mathrm{op}, 1} \| \dot{\ell}(\bbeta^0)\|_{\infty} \\
[\text{Lemma}~\ref{prop2-1}]&\leq 12\| \widehat{{\bTheta}} - {\bSigma}^{-1} \|_{\infty} \|\dot{\ell}(\bbeta^0)\|_{\infty}\max_{j=1,\ldots,p} r_j \\
[\eqref{prop2.1}] &\leq 24\lambda_n\left\|{\bSigma}^{-1}\right\|_{\mathrm{op},1} \|\dot{\ell}(\bbeta^0)\|_{\infty}\max_{j=1,\ldots,p} r_j.
\end{split}
\end{equation}
According to Lemma \ref{inilasso}, we know that
$\|\dot{\ell}(\bbeta^0) \|_{\infty}=O_p\left(\sqrt{\frac{\log p}{n}}\right).$ By assumptions {(C5)} and {(C6)}, we get
\begin{equation*}
\begin{split}
\lambda_n \left\|{\bSigma}^{-1}\right\|_{\mathrm{op},1} \|\dot{\ell}(\bbeta^0)\|_{\infty} \max_{j=1,\ldots,p} r_j & \asymp s_0 \frac{\log p}{n} \left\|\bSigma^{-1} \right\|_{\mathrm{op}, 1}^{2} \max_{j=1, \ldots, p} r_j \\
&= o_p\left(\frac{1}{\sqrt{n}} \cdot \frac{1}{n^{1 / 4}} \left\|\bSigma ^{- 1} \right\|_{\mathrm{op}, 1}^2 \max_{j = 1, \ldots ,p} r_j \right)
= o_p\left( \frac{1}{\sqrt{n \log p}}\right).
\end{split}
\end{equation*}
Therefore, $\left\|\bigl(\widehat{{\bTheta}} -
{\bSigma}^{-1}\bigr)
\dot{\ell}(\bbeta^0)\right\|_{\infty} = o_p\left( \frac{1}{\sqrt{n \log p}}\right).$
\end{proof}

\begin{proposition}\label{prop3}
Suppose conditions {(C1)--(C6)} hold.
We set
 $\lambda\asymp\sqrt{\frac{\log (p/s_0)}{n}}$ in $\eqref{lassoes}$
and get the Lasso estimator $\widehat{\bbeta}$. In $\eqref{Theta}$, we set
$$\lambda_n\asymp\left\|{\bSigma}^{-1}\right\|_{\mathrm{op}, 1}s_0\sqrt{\frac{\log p}{n}}$$
and get the estimator of ${\bSigma}^{-1}$, $\widehat{{\bTheta}}$. Then we have
$$\left\|\left(\widehat{{\bTheta}}{W}(\widetilde{\bbeta}) +
{I}_p\right)\left(\widehat{\bbeta} - \bbeta^0\right) \right\|_{\infty}
=o_p\left( \frac{1}{\sqrt{n \log p}}\right).$$
\end{proposition}
\begin{proof}
First we have decomposition
\begin{equation}\label{prop3-1}
\begin{split}
&\| \widehat{{\bTheta}}{W}(\widetilde{\bbeta}) +
{I}_p\|_{\infty}  \\
=&\left\| \left( \widehat{{\bTheta}}-{\bSigma}^{-1}\right)
\left({W}(\widetilde{\bbeta})+ {\bSigma}\right)+
{\bSigma}^{-1}\left({W}(\widetilde{\bbeta})+ {\bSigma} \right)-
\left( \widehat{{\bTheta}}-{\bSigma}^{-1}\right){\bSigma}
\right\|_{\infty}\\
\leq&\left\|\widehat{{\bTheta}}-{\bSigma}^{-1}\right\|_{\mathrm{op}, \infty}
\left\|{W}(\widetilde{\bbeta})+ {\bSigma}\right\|_{\infty}+
\left\|{\bSigma}^{-1}\right\|_{\mathrm{op}, 1}
\left\|{W}(\widetilde{\bbeta})+ {\bSigma}\right\|_{\infty}+
\left\|\widehat{{\bTheta}}-{\bSigma}^{-1}\right\|_{\mathrm{op}, \infty}
\left\|{\bSigma} \right\|_{\infty} \\
\leq&\left(24\lambda_n\max_{j\in[p]} r_j +1\right) \left\|{\bSigma}^{-1}\right\|_{\mathrm{op}, 1}
\left\|{W}(\widetilde{\bbeta})+ {\bSigma}\right\|_{\infty}+
24\lambda_n\max_{j\in[p]} r_j \left\|{\bSigma}^{-1}\right\|_{\mathrm{op}, 1} O(1) \\
\leq&\left(24\lambda_n\max_{j\in[p]} r_j +1\right) \left\|{\bSigma}^{-1}\right\|_{\mathrm{op}, 1}
\left\{ \left\|{W}(\widetilde{\bbeta})-\ddot{\ell}(\bbeta^0)\right\|_{\infty}+
\left\|\ddot{\ell}(\bbeta^0)+{\bSigma}\right\|_{\infty}
\right\} \\
&+24\lambda_n\max_{j\in[p]} r_j \left\|{\bSigma}^{-1}\right\|_{\mathrm{op}, 1} O(1).
\end{split}
\end{equation}
The first part of the second inequality in above expression uses the derivation in \eqref{prop2.2}.
For the second part, note that
$${\bSigma}_{jk}=\mE[ x_{ij}x_{ik}\ddot{b}(\bx_i^T\bbeta^0)],
\text{where}~j,k\in[p].$$
So by the condition {(C1),(C2),(C4)(i)}, we have $\left\|{\bSigma} \right\|_{\infty}=O(1).$

For the last two lines of $\eqref{prop3-1}$, on one hand we have
\begin{equation*}
\left\|{W}(\widetilde{\bbeta})-\ddot{\ell}(\bbeta^0)\right\|_{\infty}
=\max_{j\in[p]} \left\| \ddot{\ell}_j(\widetilde{\bbeta}_j)-\ddot{\ell}_j(\bbeta^0)
\right\|_{\infty},	
\end{equation*}
whose $(k, l)$-th $(k,l\in[p])$ element is
\begin{equation}\label{prop3-2}
\begin{split}
\left | \frac{1}{n}\sum_{i=1}^n x_{ik}x_{il} \left[\ddot{b}(\bx_i^T\bbeta^0)
-\ddot{b}(\bx_i^T\widetilde{\bbeta}) \right]\right|&\leq \frac{1}{n}\sum_{i=1}^n \left|x_{ik}x_{il}\right|
\cdot C_R\left| \bx_i^T\bbeta^0-\bx_i^T\widetilde{\bbeta}\right|  \\
&\leq \frac{1}{n}\sum_{i=1}^n C^3\|\widetilde{\bbeta}-\bbeta^0\|_1 \\
&\leq C^3\|\widehat{\bbeta}-\bbeta^0\|_1= O_p\left(s_0\sqrt{\frac{\log (p/s_0)}{n}} \right).
\end{split}
\end{equation}
The first inequality here uses the Lipschitz continuity property of the condition {(C4)(ii)}. The third inequality is due to the definition of $\widetilde{\bbeta}$ in \eqref{Taylor},
and the last equation is due to Theorem \ref{lassobeta1} .

On the other hand, from \eqref{ell-5} we know that
\begin{equation}\label{prop3-3}
\left\|\ddot{\ell}(\bbeta^0)+{\bSigma}\right\|_{\infty}=O_p\left(\sqrt{\frac{\log p}{n}} \right).
\end{equation}
So by \eqref{prop3-1}, \eqref{prop3-2}, \eqref{prop3-3}, we have
\begin{equation*}
\| \widehat{{\bTheta}}{W}(\widetilde{\bbeta}) +
{I}_p\|_{\infty} = O_p \left(s_0 \sqrt{\frac{\log p}{n}} \cdot \left\|{\bSigma}^{-1}\right\|_{\mathrm{op}, 1}^2 \max_{j\in[p]} r_j \right).
\end{equation*}

Finally by {(C6)} and Theorem \ref{lassobeta1},
\begin{equation*}
\begin{split}
\left\|\left(\widehat{{\bTheta}}{W}(\widetilde{\bbeta}) + {I}_p\right)\left(\widehat{\bbeta} - \bbeta^0\right) \right\|_{\infty}
\leq \left \| \widehat{{\bTheta}}{W}(\widetilde{\bbeta}) + {I}_p \right\|_{\infty} \| \widehat{\bbeta} - \bbeta^0 \|_1 & = O_p \left(s_0^2 \frac{\log p}{n} \cdot \left\|{\bSigma}^{-1}\right\|_{\mathrm{op}, 1}^2 \max_{j\in[p]} r_j \right) \\
& = o_p\left( \frac{1}{\sqrt{n \log p}}\right)
\end{split}
\end{equation*}
\end{proof}

\subsection{Complete the proof of Proposition \ref{main-thm1}}
\begin{proof}
For the first part of Proposition \ref{main-thm1}, by \eqref{decom} we have
\begin{equation}\label{decom-f}
\begin{split}
\sqrt{n}\bc^T(\widehat{\bbeta}^d - \bbeta^0)
=&\sqrt{n}\bc^T {\bSigma}^{-1} \dot{\ell}\left(\bbeta^{0}\right)+
\sqrt{n}\bc^T\left(\widehat{{\bTheta}}-
{\bSigma}^{-1}\right) \dot{\ell}\left(\bbeta^{0}\right)\\
&+\sqrt{n}\bc^T\left\{\widehat{{\bTheta}} {W}(\widetilde{\bbeta})
+{I}_p\right\}\left(\widehat{\bbeta}-\bbeta^{0}\right).
\end{split}
\end{equation}
In \eqref{decom-f}, using $\|\bc \|_1=1$, Proposition \ref{prop2} and Proposition \ref{prop3},  we know that
\begin{align*}
\sqrt{n}\bc^T\left(\widehat{{\bTheta}}-
{\bSigma}^{-1}\right) \dot{\ell}\left(\bbeta^{0}\right) &= o_p( 1),\\
\sqrt{n}\bc^T\left\{\widehat{{\bTheta}} {W}(\widetilde{\bbeta})
+{I}_p\right\}\left(\widehat{\bbeta}-\bbeta^{0}\right) &=o_p(1).	
\end{align*}

By Proposition \ref{prop1} and Slutsky's theorem \citep{Shao2003} we know that
\begin{equation*}
\sqrt{n}\bc^T(\widehat{\bbeta}^d - \bbeta^0)
\stackrel{d}{\to} \mathcal{N}(0, \sigma^2).	
\end{equation*}

For the second part of Corollay \ref{main-thm1}, denote
$a = \sqrt{n}\bc^T(\widehat{\bbeta}^d - \bbeta^0),$ and then
\begin{equation}\label{thm1-ex0}
\begin{split}
\frac{a}
{\sqrt{\bc^T\widehat{{\bTheta}}\bc}} &=
\frac{a}{\sigma} \cdot
\frac{\sigma}{\sqrt{\bc^T\widehat{{\bTheta}}\bc}}
= \frac{a}{\sigma}\left(1+ \frac{\sigma-
\sqrt{\bc^T\widehat{{\bTheta}}\bc}}
{\sqrt{\bc^T\widehat{{\bTheta}}\bc}} \right) \\
&= \frac{a}{\sigma}
\left(1+ \frac{\sigma^2-
\bc^T\widehat{{\bTheta}}\bc}
{\sqrt{\bc^T\widehat{{\bTheta}}\bc}
\left( \sigma+\sqrt{\bc^T\widehat{{\bTheta}}\bc}\right)
} \right).
\end{split}
\end{equation}
The condition of Proposition \ref{main-thm1} gives $\bc^T{\bSigma}^{-1}\bc
\rightarrow \sigma^2 \in (0,\infty)$, and thus
\begin{equation*}
\begin{split}
\left|\bc^T\widehat{{\bTheta}}\bc-
\bc^T{\bSigma}^{-1}\bc \right|
&=\left|\bc^T \left(\widehat{{\bTheta}}-
{\bSigma}^{-1}\right)c \right|
\leq \left\|\widehat{{\bTheta}}-
{\bSigma}^{-1}\right\|_{\infty} \\
&\leq 2 \lambda_n\left\|{\bSigma}^{-1}\right\|_{\mathrm{op},1}
=o_p(1),
\end{split}
\end{equation*}
where the second inequality is due to \eqref{prop2.1}, and the last equation uses the conditions {(C5)(ii)} and {(C6)}. Therefore,
\begin{equation*}
\frac{\sigma^2-\bc^T\widehat{{\bTheta}}\bc}{\sqrt{\bc^T\widehat{{\bTheta}}\bc}\left( \sigma+\sqrt{\bc^T\widehat{{\bTheta}}\bc}\right)} =o_p(1).	
\end{equation*}
Finally substitute it into \eqref{thm1-ex0}, from Slutsky's theorem, we know that
\begin{equation*}
\frac{a}
{\sqrt{\bc^T\widehat{{\bTheta}}\bc}} =
\frac{a}{\sigma}\left(1+o_p(1) \right)
 \stackrel{d}{\to} \mathcal{N}(0, 1),
\end{equation*}
which completes the proof.
\end{proof}

\section{Proofs of Theorem \ref{main-thm2} and Theorem \ref{main-thm3}}
\subsection{Construct Intermediate Variables}
To derive properties of statistic $T_j, j \in [p]$ defined in \eqref{standard1}, we construct statistics
\begin{equation}\label{standard2}
\begin{split}
T'_j &=\frac{\sqrt{n} (\widehat{\beta}^d_j-\beta^0_j)}{\sqrt{\widehat{\Theta}_{jj}}},~~\widetilde{T}_j=\frac{1}{\sqrt{n\widehat{\Theta}_{jj}}}
 \sum_{m=1}^{n} \bTheta_j^T \bx_m\left(y_m-\dot{b}
(\bx_m^T\boldsymbol{\beta}^0)\right),\\
\check{T}_j&=\frac{1}{\sqrt{n\Theta_{jj}}}
 \sum_{m=1}^{n} \bTheta_j^T \bx_m\left(y_m-\dot{b}
(\bx_m^T\boldsymbol{\beta}^0)\right).
\end{split}
\end{equation}
Then, for $1\leq j,k\leq p$,
\begin{equation}\label{checkTmoment}
\begin{split}
\mE\check{T}_j & =\frac{1}{\sqrt{n\Theta_{jj}}}\sum_{m=1}^{n}
\mE \left\{\bTheta_j^T \bx_m\mE\left(y_m-\dot{b}
(\bx_m^T\boldsymbol{\beta}^0)\bigg| \bx_m\right)\right\}=0,\\
\mE\check{T}^2_j&=\frac{1}{n\Theta_{jj}}\sum_{m=1}^{n}
\mE \left\{\bTheta_j^T \bx_m\bx_m^T\bTheta_j\mE\left[\left(y_m-\dot{b}
(\bx_m^T\boldsymbol{\beta}^0)\right)^2\bigg| \bx_m\right]\right\}\\
&=\frac{1}{n\Theta_{jj}}\sum_{m=1}^{n}
\mE \left\{\bTheta_j^T \bx_m\bx_m^T
\ddot{b}(\bx_m^T\boldsymbol{\beta}^0)\bTheta_j\right\}=1,\\
\mE(\check{T}_j\cdot\check{T}_k) &=
\frac{1}{n\sqrt{\Theta_{jj}\Theta_{kk}}}\sum_{m=1}^{n}
\mE \left\{\bTheta_j^T \bx_m\bx_m^T\bTheta_k\mE\left[\left(y_m-\dot{b}
(\bx_m^T\boldsymbol{\beta}^0)\right)^2\bigg| \bx_m\right]\right\} \\
&=\frac{1}{n\sqrt{\Theta_{jj}\Theta_{kk}}}\sum_{m=1}^{n}
\mE \left\{\bTheta_j^T \bx_m\bx_m^T
\ddot{b}(\bx_m^T\boldsymbol{\beta}^0)\bTheta_k\right\}=\Theta_{jk}^0.
\end{split}
\end{equation}
Proposition \ref{prop1} indicates as $n,p\rightarrow \infty$, $\check{T}_j \stackrel{d}{\to} \mathcal{N}(0, 1)$ by taking $\bc = \bm{1}_p$.

The following lemma is similar to Lemma 7 \citep{ma2020global}, which illustrates $T'_j, \widetilde{T}_j$ and $\check{T}_j$ are close to each other.
\begin{lemma}\label{cailemma7}
If conditions for Proposition \ref{main-thm1} and {(C8)} hold, then
\begin{equation*}
\max_{j\in [p]}|T'_j-\widetilde{T}_j|=o_p\left(\frac{1}{\sqrt{\log p}}\right),
\max_{j\in [p]}|\widetilde{T}_j-\check{T}_j|=o_p\left(\frac{1}{\sqrt{\log p}}\right).
\end{equation*}
\end{lemma}
\begin{proof}
From \eqref{prop2.1},  we know that
\begin{center}
$\|\widehat{\bTheta} - \bTheta \|_{\infty}
\leq\lambda_n\left\|\bSigma^{-1}\right\|_{\mathrm{op},1}$ and $\lambda_n\asymp\left\|\bSigma^{-1}\right\|_{\mathrm{op}, 1}s_0\sqrt{\frac{\log p}{n}}$.
\end{center}
 Conditions {(C5)} and {(C6)} tell us
\begin{equation}\label{thetaomega}
\max_{j\in [p]}|\widehat{\Theta}_{jj}-\Theta_{jj}|=
o_p\left(\frac{1}{n^{1/4}}\right).
\end{equation}
Remark under the Condition (C8)  and \eqref{thetaomega} imply that there exist constants $0<\tilde{c}_1<\tilde{c}_2$, so that as $p\rightarrow \infty$,
\begin{equation}\label{theta_hat}
\mP\left(\max_{j\in [p]}\widehat{\Theta}_{jj} \in [\tilde{c}_1,\tilde{c}_2]\right) \rightarrow 1.
\end{equation}
Also, from \eqref{decom}, \eqref{maindecom}, Proposition \ref{prop2} and Proposition \ref{prop3}, we know that
\begin{equation*}
\sqrt{n}(\boldsymbol{\widehat{\beta}}^d-\boldsymbol{\beta}^0)=\frac{1}{\sqrt{n}} \sum_{m=1}^{n} \bTheta \bx_m\left(y_m-\dot{b} (\bx_m^T\boldsymbol{\beta}^0)\right)+\Delta,
\end{equation*}
with $\|\Delta\|_{\infty}=o_p(\frac{1}{\sqrt{\log p}}).$ Hence, combining it with \eqref{theta_hat}, we know that
\begin{equation*}
\max_{j\in [p]}|T'_j-\widetilde{T}_j|=\max_{j\in [p]}\left|\frac{\Delta_j}{\sqrt{\widehat{\Theta}_{jj}}} \right|
=o_p\left(\frac{1}{\sqrt{\log p}}\right).
\end{equation*}

For the second part of Lemma, use Lemma \ref{maxi} with $X_{mj}=\frac{1}{\sqrt{n\Theta_{jj}}}\bTheta_j^T \bx_m\left(y_m-\dot{b}
(\bx_m^T\boldsymbol{\beta}^0)\right)$. Then, $\check{T}_j=\sum_{m=1}^{n}X_{mj}$ and
\begin{equation}\label{ETT1}
\begin{split}
\mE\left[\max_{j\in [p]}|\check{T}_j|\right]=
\mE\left[\max _{j\in [p]}\left|\sum_{m=1}^{n} X_{mj}\right|\right] &\leq 4(\log p)^{1 / 2} \max _{j\in [p]} \mE\left[\left(\sum_{m=1}^{n} X_{mj}^{2}\right)^{1 / 2}\right]\\
&\leq [2\log(2p)]^{1 / 2} \max _{j\in [p]}\left[ \mE\left(\sum_{m=1}^{n} X_{mj}^{2}\right)\right]^{1 / 2}\\
&=[2\log(2p)]^{1 / 2}\sum_{m=1}^{n}\frac{\mE\bTheta_j^T\bx_m\bx_m^T(y_m-\dot{b}(\bx_m^T\boldsymbol{\beta}^0))^2 \bTheta_j}{n\Theta_{jj}}\\
&=[2\log(2p)]^{1 / 2},
\end{split}
\end{equation}
where the last equality is because of $\bSigma:=\mE\bx_i\bx_i^T\ddot{b}(\bx_i^T\bbeta^0).$
By \eqref{ETT1} and Markov's inequality,
\begin{equation}\label{checkT}
\max_{j\in [p]}|\check{T}_j|=O_p(\sqrt{\log p}).
\end{equation}
Thus, with \eqref{thetaomega}, \eqref{theta_hat}, \eqref{checkT}, Remark under the Condition (C8) and the equation
\begin{equation*}
\begin{split}
|\widetilde{T}_j-\check{T}_j| &=|\check{T}_j|\cdot\left|1-
\sqrt{\frac{\widehat{\Theta}_{jj}}{\Theta_{jj}}} \right|=|\check{T}_j|\cdot\frac{|\Theta_{jj}-\widehat{\Theta}_{jj}|}
{\sqrt{\Theta_{jj}}\left(\sqrt{\Theta_{jj}}+\sqrt{\widehat{\Theta}_{jj}}\right)},
\end{split}
\end{equation*}
we have
\begin{equation*}
\max_{j\in [p]}|\widetilde{T}_j-\check{T}_j|=o_p\left(\frac{1}{\sqrt{\log p}}\right).
\end{equation*}
Here, we use $p=o(e^{n^{1 / 4}})$ assumed in {(C5)(i)}.
\end{proof}

\subsection{Proof of Theorem \ref{main-thm2}}
\begin{proof}
Partly drawing on the proof of Theorem 3.1 of \cite{javanmard2019false} and Theorem 4 of \cite{ma2020global}, we shall prove Theorem \ref{main-thm2}.

Denote $S_{\leq0}\equiv\{j\in[p]:\beta^0_j\leq 0\}, S_{\geq 0}\equiv\{j\in[p]:\beta^0_j\geq 0\},   p_1=|S_{\leq 0}|, p_2=|S_{\geq 0}|$ and $\widehat{S}\equiv\{j\in[p]:|T_j|\geq t\}$. For all $t>0$, we define
\begin{align*}
\textup{FDP}_\textup{d}(t) &= \frac{\sum_{j\in S_{\geq 0}} \mathbbm{1}(T_{j} \leq -t)
+\sum_{j\in S_{\leq0}} \mathbbm{1}( T_{j} \geq t)}
{\max \big\{\sum_{j=1}^p \mathbbm{1}( |T_{j}| \geq t),1 \big\}}, \quad\quad \textup{FDR}_\textup{d}(t)=\mE[\textup{FDP}_\textup{d}(t)],\\
\textup{FDV}_\textup{d}(t) &=\mE\left\{ \sum_{j\in S_{\geq0}} \mathbbm{1}( T_{j} \leq -t)
+\sum_{j\in S_{\leq0}} \mathbbm{1}( T_{j} \geq t) \right\},\\
\textup{FWER}_\textup{d}(t) &=\mP\left\{ \left[\sum_{j\in S_{\geq0}} \mathbbm{1}( T_{j} \leq -t)
+\sum_{j\in S_{\leq0}} \mathbbm{1}( T_{j} \geq t)\right]\geq1 \right\}.
\end{align*}

We first consider the situation where $t_0$ in GMT procedure does not exist. For this case, we take $t_0=\sqrt{2\log p}.$ For the numerator of $\textup{FDP}_\textup{d}(t)$, we have
\begin{equation}\label{thm2-1}
\begin{split}
&\mP\left\{\left[\sum_{j\in S_{\geq0}} \mathbbm{1}( T_{j} \leq -\sqrt{2\log p})
+\sum_{j\in S_{\leq0}} \mathbbm{1}( T_{j} \geq \sqrt{2\log p}) \right]\geq 1\right\}\\
&\leq
\mP\left\{\sum_{j\in S_{\leq0}} \mathbbm{1}( T_{j} \geq \sqrt{2\log p}) \geq 1\right\}+
\mP\left\{\sum_{j\in S_{\geq0}} \mathbbm{1}( T_{j} \leq -\sqrt{2\log p})\geq 1\right\}.
\end{split}
\end{equation}
For the first part of \eqref{thm2-1}, we have
\begin{equation*}
\begin{split}
&\mP\left\{\sum_{j\in S_{\leq0}} \mathbbm{1}( T_{j} \geq \sqrt{2\log p}) \geq 1\right\}
 \leq \mP\left\{\sum_{j\in S_{\leq0}} \mathbbm{1}( T'_{j} \geq \sqrt{2\log p}) \geq 1\right\}\nonumber\\
\leq
&\mP\left\{\sum_{j\in S_{\leq0}} \mathbbm{1}\left( \check{T}_{j} \geq \sqrt{2\log p}-
\frac{1}{\sqrt{\log p}}\right) \geq 1\right\}+
\mP\left\{\|T'_{j}-\check{T}_{j}\|_{\infty}\geq\frac{1}{\sqrt{\log p}} \right\}\nonumber\\
\leq &|S_{\leq0}| \cdot \mP\left\{\check{T}_{j} \geq \sqrt{2\log p}- \frac{1}{\sqrt{\log p}} \right\} +o(1).
\end{split}
\end{equation*}
The first inequality is due to the fact that $T_j' \geq T_j$ for $j \in S_{\leq 0}$ and the last inequality comes from Lemma \ref{cailemma7}. Now, for $\check{T}_{j}$, utilizing Lemma 6.1 of \cite{liu2013gaussian} with $\xi_m=\frac{\bTheta_j^T \bx_m\left(y_m-\dot{b} (\bx_m^T\boldsymbol{\beta}^0)\right)}{\sqrt{\Theta_{jj}}}, m\in[n]$, we have
\begin{equation}\label{thm2-2}
\sup_{0\le t\le 2\sqrt{\log p}}\bigg| \frac{P(|\check{T}_j| \ge t)}{G(t)}-1 \bigg|\le C(\log p)^{-1}.
\end{equation}
With the help of Lemma \ref{Gt1} and Lemma \ref{Gt2}, we then obtain
\begin{equation*}
\begin{split}
\mP\left\{\check{T}_{j} \geq \sqrt{2\log p}- \frac{1}{\sqrt{\log p}} \right\}
\leq & G(\sqrt{2\log p})\cdot[1+C(\log p)^{-1}] \\
\leq & \sqrt{\frac{2}{\pi}}\cdot \frac{1}{p\sqrt{2\log p}}\cdot[1+C(\log p)^{-1}].
\end{split}
\end{equation*}
Substitute this into \eqref{thm2-2} and then we get $\mP\left\{\sum_{j\in S_{\leq0}} \mathbbm{1}( T_{j} \geq \sqrt{2\log p}) \geq 1\right\}=O\left(\frac{1}{\sqrt{\log p}}\right).$ Note that this upper bound also holds for the second part of \eqref{thm2-1}. Then, \eqref{thm2-1} gives
\begin{equation*}
\mP\left\{\left[\sum_{j\in S_{\geq0}} \mathbbm{1}( T_{j} \leq -\sqrt{2\log p})+\sum_{j\in S_{\leq0}} \mathbbm{1}( T_{j} \geq \sqrt{2\log p}) \right] = 0\right\} = 1 - O\left(\frac{1}{\sqrt{\log p}}\right),
\end{equation*}
which means that \eqref{fdp} hold.

On the other hand, if $0\leq t_0\leq t_p$ exists, we have
\begin{equation}\label{thm2-4}
\begin{split}
\textup{FDP}_\textup{d}(t_0)= & \frac{\sum_{j\in S_{\geq0}} \mathbbm{1}( T_{j} \leq -t_0)
+\sum_{j\in S_{\leq 0}} \mathbbm{1}( T_{j} \geq t_0)}
{\max \big\{\sum_{j=1}^p \mathbbm{1}( |T_{j}| \geq t_0),1 \big\}} \\
\leq &  \frac{\sum_{j\in S_{\geq0}} \mathbbm{1}( T'_{j} \leq -t_0)
+\sum_{j\in S_{\leq 0}} \mathbbm{1}( T'_{j} \geq t_0)}
{\max \big\{\sum_{j=1}^p \mathbbm{1}( |T_{j}| \geq t_0),1 \big\}} \\
= &  \frac{\sum_{j\in S_{\geq0}}\left[ \mathbbm{1}( T'_{j} \leq -t_0)-Q(t_0)\right]
+\sum_{j\in S_{\leq 0}} \left[\mathbbm{1}( T'_{j} \geq t_0)-Q(t_0)\right]
+(p_1+p_2)Q(t_0)}
{\max \big\{\sum_{j=1}^p \mathbbm{1}( |T_{j}| \geq t_0),1 \big\}}\\
\leq & \frac{pG(t_0)\left(\frac{p_1+p_2}{2p}+A_p \right)}
{\max \big\{\sum_{j=1}^p \mathbbm{1}( |T_{j}| \geq t_0),1 \big\}}
\leq \alpha\left(1-\frac{s_0}{2p}+A_p \right).
\end{split}
\end{equation}
Here $Q(t)=\frac{G(t)}{2}$ and $t_p=\sqrt{2\log p-2\log\log p}$, and
\begin{equation}\label{thm2-5}
A_{p} = \sup _{0 \leq t \leq t_{p}} \left| \frac{\sum_{j \in S_{\geq 0}}\left[\mathbbm{1}\left(T'_{j} \leq-t\right)-Q(t)\right]+\sum_{j \in S_{\leq 0}}\left[\mathbbm{1}\left(T'_{j} \geq t\right)-Q(t)\right]}{p G(t)}\right|.
\end{equation}
If $A_p \stackrel{p}{\rightarrow} 0$, \eqref{fdp} holds, then the proof is completed. Proposition \ref{GGt} below ensures $A_p \stackrel{p}{\rightarrow} 0$.
\end{proof}
\begin{proposition} \label{GGt}
Under conditions in Theorem \rm{\ref{main-thm2}}, for $A_p$ defined in \eqref{thm2-5}, when $(n,p)\rightarrow \infty$ , we have
$A_p  \stackrel { p }\rightarrow 0.$
\end{proposition}

\begin{proof}
First, we have
\begin{equation*}
A_{p} \leq \sup _{0 \leq t \leq t_{p}} \left| \frac{\sum_{j \in S_{\geq 0}}\left[\mathbbm{1}\left(T'_{j} \leq-t\right)-Q(t)\right]}{p G(t)}\right|+
\sup _{0 \leq t \leq t_{p}} \left| \frac{\sum_{j \in S _{\leq 0}}\left[\mathbbm{1}\left(T'_{j} \geq t\right)-Q(t)\right]}{p G(t)}\right|.
\end{equation*}
Without loss of generality, it suffices to show that, for any $\varepsilon >0$,
\begin{equation*}
\mP\left(\sup _{0 \leq t \leq t_{p}} \left| \frac{\sum_{j \in S_{\leq 0}}\left[\mathbbm{1}\left(T'_{j} \geq t\right)-Q(t)\right]}{p G(t)}\right| \leq \varepsilon\right) \rightarrow 1.
\end{equation*}
Then, according to Lemma \ref{Gt2} and Lemma \ref{cailemma7}, we only need to prove
\begin{equation}\label{thm2-7}
\mP\left(\sup _{0 \leq t \leq t_{p}} \left| \frac{\sum_{j \in S_{\leq 0}}\left[\mathbbm{1}\left(\check{T}_{j} \geq t\right)-Q(t)\right]}{p G(t)}\right| \leq \varepsilon\right) \rightarrow 1.
\end{equation}
Let $z_0<z_1<\ldots<z_{b_p}\leq 1,z_0=G(t_p)$ and $\tau_i =G^{-1}(z_i)(0\leq i\leq b_p).$ Note that $G(t_p) \asymp \frac{\sqrt{\log p}}{p}$ by Lemma \ref{Gt1}. Take $b_p$ as the largest integer less than or equal to $(\log p)^{1+\delta}$, for some $0<\delta<1/3.$ Then, we have $q = \left[G(t_p)\right]^{-\frac{1}{b_p}}$ satisfies
\begin{equation*}
\begin{split}
(q-1) &\asymp \log q=-\frac{\log(\sqrt{\log p}/p) }{b_p}=O((\log p)^{-\delta}),\\
(q^{1-\rho}-1) &\asymp (1-\rho)\log q=O((\log p)^{-\delta}).
\end{split}
\end{equation*}
This means when $p \rightarrow \infty$, $q \rightarrow 1$.
Next, we take $z_i=z_0\cdot q^i(0\leq i\leq b_p)$ and obtain
\begin{equation*}
\frac{G(\tau_i)}{G(\tau_{i+1})}=\frac{z_i}{z_{i+1}}=q \to 1 ~\text{uniformly}.
\end{equation*}
Therefore, proving \eqref{thm2-7} turns into proving
\begin{equation*}
\mP\left(\max _{0\leq i\leq b_p} \left| \frac{\sum_{j \in S_{\leq 0}}\left[\mathbbm{1}\left(\check{T}_{j} \geq \tau_i\right)-Q(\tau_i)\right]}{p G(\tau_i)}\right| \geq \varepsilon\right) \rightarrow 0.
\end{equation*}
Actually, according to \eqref{thm2-2}, we now only need to show
\begin{equation*}
\mP\left(\max _{0\leq i\leq b_p} \left| \frac{\sum_{j \in S_{\leq 0}}\left[\mathbbm{1}\left(\check{T}_{j} \geq \tau_i\right)-\mP\left(\check{T}_{j} \geq \tau_i\right)\right]}{p G(\tau_i)}\right| \geq \varepsilon\right) \rightarrow 0,
\end{equation*}
which satisfies
\begin{equation*}
\begin{split}
&\mP\left(\max _{0\leq i\leq b_p} \left| \frac{\sum_{j \in {S \leq 0}}\left[\mathbbm{1}\left(\check{T}_{j} \geq \tau_i\right)-\mP\left(\check{T}_{j} \geq \tau_i\right)\right]}{p G(\tau_i)}\right| \geq \varepsilon\right) \\
\leq
&\sum_{i=0}^{b_p}\mP\left( \left| \frac{\sum_{j \in S _{\leq 0}}\left[\mathbbm{1}\left(\check{T}_{j} \geq \tau_i\right)-\mP\left(\check{T}_{j} \geq \tau_i\right)\right]}{p G(\tau_i)}\right| \geq \varepsilon\right).
\end{split}
\end{equation*}
Define $I(t)=\frac{\sum_{j \in S _{\leq 0}}\left[\mathbbm{1}\left(\check{T}_{j} \geq t\right)-\mP\left(\check{T}_{j} \geq t\right)\right]}{p G(t)}.$ By Markov's inequality, we know that $P(|I(t)\geq \varepsilon|)\leq\frac{\mE[I(t)^2]}{\varepsilon^2}$. Thence, the problem turns into showing
\begin{equation*}
\sum_{i=0}^{b_p}\mE[I(\tau_i)^2]=o(1).
\end{equation*}

In fact, utilizing the definition of $\mathcal{A}(\gamma)$, we have followimng decomposition
\begin{equation}\label{thm2-12}
\begin{split}
\mE[I(t)^2]=&\frac{\sum_{j \in S _{\leq 0}}\left[\mP\left(\check{T}_{j} \geq t\right)-P^2\left(\check{T}_{j} \geq t\right)\right]}{p^2 G^2(t)} \\
&+ \frac{\sum_{j,k \in S _{\leq 0},k\neq j}\left[\mP\left(\check{T}_{j} \geq t,\check{T}_{k} \geq t\right)-\mP\left(\check{T}_{j} \geq t\right)\mP\left(\check{T}_{k} \geq t\right)\right]}{p^2 G^2(t)}\\
\leq & \frac{C}{pG(t)}+
\frac{1}{p^2}\sum_{(j,k)\in\mathcal{A}^c(\gamma)\cap S _{\leq 0}}
\left[\frac{\mP\left(\check{T}_{j} \geq t,\check{T}_{k} \geq t\right)}{G^2(t)}-1\right]+\frac{1}{p^2}\sum_{(j,k)\in\mathcal{A}(\gamma)\cap S _{\leq 0}}
\frac{\mP\left(\check{T}_{j} \geq t,\check{T}_{k} \geq t\right)}{G^2(t)} \\
=&\frac{C}{pG(t)}+\Delta_{11}(t)+\Delta_{12}(t).
\end{split}
\end{equation}
Similar to the derivation of \eqref{thm2-2}, for $(j,k)\in\mathcal{A}^c(\gamma)$, from Lemma 6.1 in \cite{liu2013gaussian}, we know that for any $0\leq t \leq \sqrt{2\log p}$, uniformly
\begin{equation*}
\Delta_{11}(t)\leq C(\log p)^{- \frac{3}{2}}.
\end{equation*}
For $(j,k) \in \mathcal{A}(\gamma)$, by Lemma 6.2 in \cite{liu2013gaussian},
\begin{equation*}
\mP\left(|\check{T}_{j}| \geq t,|\check{T}_{k}| \geq t\right) \leq C(t+1)^{-2}\exp
\left(-\frac{t^2}{1+|\Theta^0_{jk}|} \right).
\end{equation*}
By Lemma \ref{Gt1}, we know that
$(t+1)^{-2}\exp \left(-\frac{t^2}{1+|\Theta^0_{jk}|} \right) \leq C \left[ G(t)\right]^{\frac{2}{1+|\Theta^0_{jk}|}}$, so
\begin{equation*}
\begin{split}
\Delta_{12}(t)\leq  &\frac{C}{p^2G^2(t)}\sum_{(j,k)\in\mathcal{A}(\gamma)\cap S _{\leq 0}}
(t+1)^{-2}\exp\left(-\frac{t^2}{1+|\Theta^0_{jk}|} \right) \\
\leq &\frac{C}{p^2} \sum_{(j,k)\in\mathcal{A}(\gamma)\cap S _{\leq 0}}\left[
G(t) \right]^{-\frac{2|\Theta^0_{jk}|}{1+|\Theta^0_{jk}|}}.
\end{split}
\end{equation*}
We divide $\mathcal{A}(\gamma)$ into $\mathcal{A}(\gamma)\cap\mathcal{B}(\rho)$ and $\mathcal{A}(\gamma)\cap\mathcal{B}^c(\rho).$ Therewith we have
\begin{equation*}
\begin{split}
\frac{1}{p^2} \sum_{(j,k)\in\mathcal{A}(\gamma)\cap\mathcal{B}(\rho)\cap S _{\leq 0}}\left[
G(t) \right]^{-\frac{2|\Theta^0_{jk}|}{1+|\Theta^0_{jk}|}} = &
O\left(\frac{1}{p G(t)} \right),\\
\frac{1}{p^2} \sum_{(j,k)\in\mathcal{A}(\gamma)\cap\mathcal{B}^c(\rho)\cap S _{\leq 0}}\left[G(t) \right]^{-\frac{2|\Theta^0_{jk}|}{1+|\Theta^0_{jk}|}} = &
o(p^{-1+\rho})\cdot[G(t)]^{-1+\rho}.
\end{split}
\end{equation*}
Subsitute all these pieces into \eqref{thm2-12}, and then we know that
\begin{equation*}
\begin{split}
\sum_{i=0}^{b_p}\mE[I(\tau_i)^2] \leq &C\cdot \sum_{i=0}^{b_p}
\left[ \frac{1}{pG(\tau_i)}+(\log p)^{- \frac{3}{2}}+\frac{1}{pG(\tau_i)}+o(p^{-1+\rho})\cdot[G(\tau_i)]^{-1+\rho}
\right] \\
\leq &C\cdot\sum_{i=0}^{b_p}\left[ \frac{2}{pz_i}+o(p^{-1+\rho})\cdot z_i^{-1+\rho}\right] +O((\log p)^{-1/2+\delta})\\
=&O\left( \frac{1}{p} \cdot\frac{q^{b_p}-1}{q-1}\right)+o(p^{-1+\rho})\cdot  \frac{q^{(1-\rho)b_p}-1}{q^{1-\rho}-1} +O((\log p)^{-1/2+\delta}) \\
=&O\left((\log p)^{-1/2+\delta}  \right)+o(p^{-1+\rho})\cdot \left(\frac{p}{\sqrt{\log p}} \right)^{1-\rho}\cdot (\log p)^{\delta}+O((\log p)^{-1/2+\delta})\\
=&O\left((\log p)^{-1/2+\delta})+o((\log p)^{-1/2+\delta+\rho/2}\right)=o(1),
\end{split}
\end{equation*}
where the last equality is because of $0<\rho<1/3$ and $0<\delta<1/3.$ Now we have completed the proof of Proposition \ref{GGt}.
\end{proof}

\subsection{Proof of Theorem \ref{main-thm3}}
\begin{proof}
Similar to \eqref{thm2-2}, we have
$\sup_{0\le t\le 2\sqrt{\log p}}\bigg| \frac{ \mP(\check{T}_j \ge t)}{Q(t)}-1 \bigg|\le C(\log p)^{-1}$.
Then, Lemma \ref{cailemma7} gives
\begin{equation*}
\sup_{0\le t\le 2\sqrt{\log p}}\bigg| \frac{ \mP(T'_j \ge t)}{Q(t)}-1 \bigg|\le C(\log p)^{-1}.
\end{equation*}
According to definition of $T'_j$ in \eqref{standard2}, we know that $T'_j \geq T_j$, for $j \in S_{\leq 0}$. Therefore,
\begin{equation*}
\frac{\sum_{j\in S_{\leq 0}} \mP(T_j \ge t_{\text{FDV}})}{p_1 Q(t_{\text{FDV}})}-1 \le
\left|\frac{\sum_{j\in S_{\leq 0}} \mP(T'_j \ge t_{\text{FDV}})}{p_1Q(t_{\text{FDV}})}-1 \right|\le  C(\log p)^{-1}.
\end{equation*}
Likewise, we also have
\begin{equation*}
\frac{\sum_{j\in S_{\geq 0}} \mP(T_j \le -t_{\text{FDV}})}{p_2Q(t_{\text{FDV}})}-1 \le C(\log p)^{-1}.
\end{equation*}

Recall that
\begin{equation*}
\textup{FDV}_\textup{d}(t_{\text{FDV}})=\sum_{j\in S_{\leq 0}} \mP(T_j \ge t_{\text{FDV}})+
\sum_{j\in S_{\geq 0}} \mP(T_j \le -t_{\text{FDV}}),
\end{equation*}
$p_1+p_2=2p-s_0, G(t_{\text{FDV}})=u/p,$ and $Q(t)=G(t)/2$. Then we have \eqref{fdv} holds. For \eqref{fwer},
\begin{equation}\label{thm3-6}
\begin{split}
\textup{FWER}_\textup{d}(t_{\text{FDV}})\le & \mP\left(\bigcup_{j\in S_{\geq 0}}\{T_j\leq -t_{\text{FDV}}\} \right)+
\mP\left(\bigcup_{j\in S_{\leq 0}}\{T_j\leq t_{\text{FDV}}\} \right) \\
 \le&\sum_{j\in S_{\geq 0}} \mP(T_j \le -t_{\text{FDV}})+\sum_{j\in S_{\leq 0}} \mP(T_j \ge t_{\text{FDV}}),
\end{split}
\end{equation}
which is the same as \eqref{thm3-6}. Likewise, we know that \eqref{fwer} holds.
\end{proof}

\section{Proof of Theorem \ref{power-thm}}
In this section, we shall prove the consistency of directional power. Unlike the proof of FDP related theorems, the denominator of \eqref{t-fdr} in the difinition of the data-dependent threshold $t_0$ is different from that of $\textup{Power}_\textup{d}$ in Definition \ref{fdrdir}, while it is the same as the denominator of $\textup{FDP}_\textup{d}$. Thus, to derive the property of $\textup{Power}_\textup{d}$, it is convenient to construct a new threshold $\tilde{t}$ independent of data. Note that the greater the threshold, the smaller the statistical power. Thus, the new threshold should be slightly bigger than $t_0$. The following Lemma gives the form of $\tilde{t}$ and shows that $t_0 \leq \tilde{t}$ holds with high probability.

\begin{lemma}\label{power-lemma}
Under conditions of Theorem $\rm{\ref{power-thm}}$, for any data-independent threshold
\begin{equation*}
\tilde{t}=G^{-1}\left( \frac{\alpha s_0}{p}(1-o(1)) \right),
\end{equation*}
it satisfies $t_0 \leq \tilde{t}$ with high probability.
\end{lemma}
We defer the proof of it to Section \ref{subsubsec:proofoftildetlemma}.

\begin{proof}
Let $S_+\equiv \{j\in[p]:\beta_j^0>0\}$ and $S_-\equiv \{j\in[p]:\beta_j^0<0\}$. Then, $S=S_+\cup S_-.$ For $j\in S$, we know $T_j=T'_j + \frac{\sqrt{n} \beta^0_j}{\sqrt{\widehat{\Theta}_{jj}}}$ by \eqref{standard2}.
For $\varepsilon>0$ and $\delta>0$, we define event
\begin{align} \label{power-e4}
\mathcal{D}(\varepsilon,\delta)\equiv\left\{\max_{j\in[p]}|T'_j-\check{T}_j| \leq\delta, \max_{j\in[p]} |\widehat{\Theta}_{jj}-\Theta_{jj}| \le \varepsilon\right\}.
\end{align}
Without ambiguity, we use $\mathcal{D}$ to denote $\mathcal{D}(\varepsilon,\delta)$ in the following part. When $t_0 \le \tilde{t}$, we have
\begin{equation*}
\begin{split}
&~~~~\rm{Power}_\textup{d}=\mE\left[\frac{|\{j\in \widehat{S}:\, \widehat{\sign}_j = \sign(\beta^0_j)\}|}{|S|\vee 1}\right] \\
\ge &\frac{1}{s_0}\sum_{j\in S_+} \mP\left(\{T_i\ge \tilde{t}\}\cap\mathcal{D}\right)+ \frac{1}{s_0}\sum_{i\in S_-} \mP\left(\{T_i \le -\tilde{t}\}\cap\mathcal{D}\right)\\
\ge&\frac{1}{s_0}\sum_{j\in S_+} \mP\left(\left\{\check{T}_i\ge \tilde{t}+\delta-\frac{\sqrt{n} \beta^0_j}{\sqrt{\Theta_{jj}+\varepsilon}}\right\}\cap\mathcal{D}\right) + \frac{1}{s_0}\sum_{j\in S_-} \mP\left(\left\{\check{T}_i \le -\tilde{t}-\delta-\frac{\sqrt{n} \beta^0_j}{\sqrt{\Theta_{jj}+\varepsilon}}\right\}\cap\mathcal{D}\right)\\
\geq &\frac{1}{s_0}\sum_{j\in S_+} \left\{\mP\left(\check{T}_i\ge \tilde{t}+\delta-\frac{\sqrt{n} \beta^0_j}{\sqrt{\Theta_{jj}+\varepsilon}}\right)\right\}+ \frac{1}{s_0}\sum_{j\in S_-} \left\{\mP\left(\check{T}_i \le -\tilde{t}-\delta-\frac{\sqrt{n} \beta^0_j}{\sqrt{\Theta_{jj}+\varepsilon}}\right) \right\} - \mP(\mathcal{D}^c).
\end{split}
\end{equation*}
From \eqref{thetaomega} with $\frac{1}{n^{1 / 4}} = o(\frac{1}{\sqrt{\log p}})$ and Lemma \ref{cailemma7},
\begin{equation}\label{power-1}
\lim_{(n,p)\rightarrow \infty} \mP \bigl(\mathcal{D} \bigl(\frac{1}{\sqrt{\log p}},\frac{1}{\sqrt{\log p}}\bigr)\bigr)=1.	
\end{equation}
Thus, for any $\varepsilon > 0$ and $\delta > 0$, $\mP(\mathcal{D})\rightarrow 1$ when $n, p \to \infty$.
Lemma 6.1 of \cite{liu2013gaussian} and the definition of $\tilde{t}$ gives
\begin{equation}\label{power-2}
\begin{split}
\liminf_{(n,p)\rightarrow\infty} \frac{\frac{1}{s_0} \left\{\sum_{j\in S_+} \mP\left(\check{T}_i\ge \tilde{t}-\frac{\sqrt{n} \beta^0_j}{\sqrt{\Theta_{jj}}}\right) + \sum_{j\in S_-} \mP \left(\check{T}_i \le -\tilde{t}-\frac{\sqrt{n} \beta^0_j}{\sqrt{\Theta_{jj}}}\right)\right\}}{\frac{1}{s_0}\sum_{j\in S} Q\left[G^{-1}\left(\frac{\alpha s_0}{p}\right)-\frac{\sqrt{n}|\beta_j^0|}{\sqrt{\Theta_{jj}}}\right]} \geq 1.
\end{split}
\end{equation}

Also, Lemma \ref{Gt1} tells us
\begin{equation*}
G\left(\sqrt{2\log\left(2p/(\alpha s_0)\right)}\right)
< \frac{2\phi\left(\sqrt{2\log\left(2p/(\alpha s_0)\right)} \right)}
{\sqrt{2\log\left(2p/(\alpha s_0)\right)}}
< \frac{\alpha s_0}{p}.
\end{equation*}
Then we know that
\begin{equation}\label{power-3}
G^{-1} \left(\frac{\alpha s_0}{p} \right)<\sqrt{2\log\left(2p/(\alpha s_0)\right)}	
\end{equation}
due to the fact that $G(t)$ is a monotonically decreasing function. With the assumption $|\beta_j^0|>\sqrt{\frac{8\Theta_{jj}\log(p/s_0)}{n}}$, we know that the denominator of \eqref{power-2} is bigger than some constant $C > 0$. Thus, we have
\begin{equation*}
\frac{\mP(\mathcal{D}^c)}{\frac{1}{s_0}\sum_{j\in S} Q\left[G^{-1}\left(\frac{\alpha s_0}{p}\right)-\frac{\sqrt{n}|\beta_j^0|}{\sqrt{\Theta_{jj}}}\right]} \rightarrow 0,
\end{equation*}
for any $\varepsilon, \delta > 0$. Combining this with \eqref{power-2} and taking sufficiently small $\varepsilon$ and $\delta$, by then we have completed the proof.

\subsection{Proof of Corollary \ref{power-cor2}}
Let $\beta_{\min}=\min_{j\in S}|\beta_j^0|$. Because $Q(t)$ is a strictly monotonically decreasing function of $t$,  it can be derived from the theorem \ref{power-thm} that
\begin{equation*}
\liminf_{n\rightarrow \infty} \frac{\rm{Power}_\textup{d}}{
Q\left[G^{-1}(\frac{\alpha
s_0}{p})-\frac{\sqrt{n}\beta_{\min}}{\sqrt{\Theta_{jj}}}\right]} \geq 1.
\end{equation*}
Thus, it suffices to show
$G^{-1}(\frac{\alpha s_0}{p})-\frac{\sqrt{n}\beta_{\min}}{\sqrt{\Theta_{jj}}}\rightarrow -\infty.$ Based on \eqref{power-3} and the assumption $\beta_{\min}$,
\begin{equation*}
G^{-1}(\frac{\alpha s_0}{p})- \frac{\sqrt{n}\beta_{\min}}{\sqrt{\Theta_{jj}}} < \sqrt{2\log\left(2p/(\alpha s_0)\right)}-\frac{\sqrt{n}\beta_{\min}}{\sqrt{\Theta_{jj}}}
\rightarrow -\infty.
\end{equation*}
So the proof is completed.
\end{proof}

\subsection{Proof of Lemma \ref{power-lemma}}\label{subsubsec:proofoftildetlemma}
\begin{proof}
First, we utilize Lemma \ref{Gt1} and obtain that
\begin{equation*}
G(\sqrt{2\log (p/s_0)}) \le \sqrt{\frac{2}{\pi}}\cdot \frac{e^{-\log (p/s_0)}}{\sqrt{2\log (p/s_0)}} = \frac{s_0}{p\sqrt{\pi \log (p/s_0)}} \le \frac{\alpha s_0}{p}(1-o(1)).
\end{equation*}
As $G(\tilde{t}) = \frac{\alpha s_0}{p}(1-o(1))$ and $G(t)$ is strictly monotonically decreasing, $\tilde{t} \leq \sqrt{2\log(p/s_0)}.$

If $t_0 > \tilde{t}$, we know from the definition of $t_0$ in GMT procedure that
\begin{equation}\label{power-e3}
\frac{pG(\tilde{t})}{\max \big\{\sum_{j=1}^p \mathbbm{1}\{ |T_{j}| \ge \tilde{t}\},1   \big\}} > \alpha.
\end{equation}
On the other hand, considering the definition of \eqref{power-e4}, on event $\mathcal{D}(\varepsilon,\delta)$, for $j \in S$, we have
\begin{equation*}
\frac{\sqrt{n} |\beta^0_j|}{\sqrt{\Theta_{jj}+\varepsilon}}
-\delta\le |T_j-\check{T}_j|\le \frac{\sqrt{n} |\beta^0_j|}{\sqrt{\Theta_{jj}-\varepsilon}}
+\delta.
\end{equation*}
And
\begin{equation*}
\begin{split}
&\mP\left(|T_j|< \tilde{t} \right) \\
< &\max\left\{\mP\left(\check{T}_j < -\frac{\sqrt{n} |\beta^0_j|}{\sqrt{\Theta_{jj}+\varepsilon}}
+\delta+\tilde{t}\right), \mP\left(\check{T}_j> \frac{\sqrt{n} |\beta^0_j|}{\sqrt{\Theta_{jj}+\varepsilon}}
-\delta-\tilde{t}\right)\right\}+
P(\mathcal{D}^c(\varepsilon, \delta)).
\end{split}
\end{equation*}
Because of $|\beta_j^0|>\sqrt{\frac{8\Theta_{jj}\log(p/s_0)}{n}}$ and Remark under the Condition (C8), we know that we can take sufficiently small $\delta$ and $\varepsilon$ so that
$$\frac{\sqrt{n} |\beta^0_j|}{\sqrt{\Theta_{jj}+\varepsilon}}
-\delta> 2\sqrt{2\log(p/s_0)} \ge 2 \tilde{t}.$$
So from Lemma 6.1 of \cite{liu2013gaussian} we know that for $j\in S$, we have
\begin{equation*}
\mP\left(|T_j|< \tilde{t} \right)< G(\tilde{t})\left(1 / 2 +o(1) \right) + \mP(\mathcal{D}^c(\varepsilon,\delta))<\frac{\alpha s_0}{2 p}(1 + o(1)) + \mP(\mathcal{D}^c(\varepsilon,\delta)).
\end{equation*}
We take a deterministic series $a_p=\mP\left(|T_j|< \tilde{t} \right)$. Utilizing \eqref{power-1}, we know that $a_p \rightarrow 0$. We can take another series $b_p \rightarrow \infty$, so that $a_p b_p \rightarrow 0.$ By Markov's inequality, we know that
\begin{equation*}
\mP\left(\sum_{j\in S} \mathbbm{1}\{ |T_{j}| < \tilde{t}\} >s_0a_pb_p\right)\le
\frac{\mE\sum_{j\in S} \mathbbm{1}\{ |T_{j}| < \tilde{t}\}}{s_0a_pb_p}
=\frac{1}{b_p}.
\end{equation*}
So $\sum_{j\in S} \mathbbm{1}\{ |T_{j}| < \tilde{t}\} \le s_0 a_p b_p$ holds with high probability. Hence, with high probability,
\begin{equation}\label{power-e10}
\begin{split}
\sum_{j=1}^p \mathbbm{1}\{ |T_{j}| \ge \tilde{t}\} \ge \sum_{j\in S} \mathbbm{1}\{ |T_{j}| \ge \tilde{t}\}=s_0-\sum_{j\in S} \mathbbm{1}\{ |T_{j}| < \tilde{t}\}\ge s_0(1-a_p b_p).
\end{split}
\end{equation}
According to \eqref{power-e3}, we know that
\begin{equation*}
\sum_{j=1}^p \mathbbm{1}\{ |T_{j}| \ge \tilde{t}\}<\frac{pG(\tilde{t})}{\alpha}.
\end{equation*}
If we take $\tilde{t}=G^{-1}\left( \frac{\alpha s_0}{p}(1-2a_p b_p)\right)$, then $\sum_{j=1}^p \mathbbm{1}\{ |T_{j}| \ge \tilde{t}\} < s_0(1-a_p b_p)$. This conflicts with \eqref{power-e10}! As a result, $t_0 \leq t$ with high probability.
\end{proof}

\section{Proof of the Results of the Two-Sample Problem}
Similar to \eqref{standard2}, for the two-sample problem, for $j \in [p]$, we define three statistics
\begin{equation*}
\begin{split}
M'_j =&\frac{\left(\widehat{\beta}^{(1)}_j-\widehat{\beta}^{(2)}_j\right)
-\left(\beta^{(1)}_j-\beta^{(2)}_j\right)}
{\sqrt{\widehat{\Theta}_{jj}^{(1)}/n_1+\widehat{\Theta}_{jj}^{(2)}/n_2}},\\
\widetilde{M}_j=&\frac{\frac{1}{n_1}\sum_{m=1}^{n_1} (\bTheta^{(1)}_j)^T x^{(1)}_m\left\{y^{(1)}_m-\dot{b}
\left((x^{(1)}_m)^T\boldsymbol{\beta}^{(1)}\right)\right\}-
\frac{1}{n_2}\sum_{m=1}^{n_2} (\bTheta^{(2)}_j)^T x^{(2)}_m\left\{y^{(2)}_m-\dot{b}
\left((x^{(2)}_m)^T\boldsymbol{\beta}^{(2)}\right)\right\}}
{\sqrt{\widehat{\Theta}_{jj}^{(1)}/n_1+\widehat{\Theta}_{jj}^{(2)}/n_2}},\\
\check{M}_j=&\frac{\frac{1}{n_1}\sum_{m=1}^{n_1} (\bTheta^{(1)}_j)^T x^{(1)}_m\left\{y^{(1)}_m-\dot{b}
\left((x^{(1)}_m)^T\boldsymbol{\beta}^{(1)}\right)\right\}-
\frac{1}{n_2}\sum_{m=1}^{n_2} (\bTheta^{(2)}_j)^T x^{(2)}_m\left\{y^{(2)}_m-\dot{b}
\left((x^{(2)}_m)^T\boldsymbol{\beta}^{(2)}\right)\right\}}
{\sqrt{\Theta_{jj}^{(1)}/n_1+\Theta_{jj}^{(2)}/n_2}}.
\end{split}
\end{equation*}
Assuming that the conditions of Lemma \ref{cailemma7} hold for both sets of samples and repeating the proof of Lemma \ref{cailemma7}, we can also prove
\begin{equation}\label{MMdif}
\|M'_j-\widetilde{M}_j\|_{\infty}=o_p\left(\frac{1}{\sqrt{\log p}}\right),
\|\widetilde{M}_j-\check{M}_j\|_{\infty}=o_p\left(\frac{1}{\sqrt{\log p}}\right).
\end{equation}
Similar to \eqref{checkTmoment}, we know that
\begin{equation}\label{checkT2}
\mE \check{M}_j=0,\mE \check{M}_j^2=1.
\end{equation}
And for $1\leq j<k\leq p$,
\begin{equation}\label{checkT3}
\mE \check{M}_j\check{M}_k=\frac{\Theta^{(1)}_{jk}/n_1+\Theta^{(2)}_{jk}/n_2}
{\sqrt{\Theta_{jj}^{(1)}/n_1+\Theta_{jj}^{(2)}/n_2}\cdot
\sqrt{\Theta_{kk}^{(1)}/n_1+\Theta_{kk}^{(2)}/n_2}}.
\end{equation}

Therewith, we know that $\mE \check{M}_j\check{M}_k=\widetilde{\Theta}_{jk}^0$ by \eqref{checkT3}. Consider \eqref{thm2-1} and \eqref{thm2-4} in the proof of Theorem \ref{main-thm2}. Actually, if we replace $S_{\le 0}$ and $S_{\ge 0}$ in \eqref{thm2-1} and \eqref{thm2-4} with $\widetilde{S}_{\le 0},\widetilde{S}_{\ge 0}$ in the two-sample problem, the inequality still holds. Then we replace $T'_j, \widetilde{T}_j$ and $\check{T}_j$ with $M'_j, \widetilde{M}_j$ and $\check{M}_j$. Compare \eqref{MMdif},\eqref{checkT2} and \eqref{checkT3} with Lemma \ref{cailemma7} and \eqref{checkTmoment}, and we know that the properties of $M'_j,\widetilde{M}_j$ and $\check{M}_j$ are exactly the same as that of $T'_j,\widetilde{T}_j$ and $\check{T}_j$. So, in order to prove Theorem \ref{twosample-thm1} and Theorem \ref {twosample-thm2}, we just need to repeat the proof of Theorem \ref{main-thm2}, Proposition \ref{GGt} and Theorem \ref{main-thm3}. The specific details will not be addressed.

\end{document}